\documentclass[11pt, twoside]{article}
\pdfoutput=1
\usepackage[margin=.75in]{geometry}

\usepackage{graphicx}
\usepackage{amsmath}
\usepackage{amssymb,amsfonts}
\usepackage{amsthm}
\usepackage{cite}
\usepackage{hyperref}
\usepackage{bm}
\usepackage{mathrsfs}
\usepackage{amssymb}
\usepackage{afterpage}
\usepackage{xcolor,enumitem,enumerate}
\usepackage{float}
\usepackage{caption}
\usepackage{subcaption}

\usepackage{hyperref}
\usepackage[nameinlink]{cleveref}

\newtheorem{theorem}{Theorem}[section]

\newtheorem{prop}[theorem]{Proposition}
\newtheorem{corollary}[theorem]{Corollary}

\newtheorem{lem}[theorem]{Lemma}

\theoremstyle{definition}

\newcommand{\real}[0]{\mathbb{R}}
\newcommand{\p}[0]{\partial}
\newcommand{\pn}[0]{\partial^-}
\newcommand{\pp}[0]{\partial^+}
\DeclareMathOperator{\Span}{span}

\theoremstyle{remark}
\newtheorem{remark}[theorem]{Remark}

\usepackage{fancyhdr}
\title{A New Proper Orthogonal Decomposition Method with Second Difference Quotients for the Wave Equation}

\begin{document}
	
	\author{Andrew Janes%
		\thanks{Department of Mathematics and Statistics, Missouri University of Science and Technology, Rolla, MO (\mbox{acjgcc@umsystem.edu}, \mbox{singlerj@mst.edu}).}
		\and
		John~R.~Singler%
		\footnotemark[1]
	}
	\maketitle
	
	\begin{abstract}
		Recently, researchers have investigated the relationship between proper orthogonal decomposition (POD), difference quotients (DQs), and pointwise in time error bounds for POD reduced order models of partial differential equations. In a recent work (Eskew and Singler, Adv.\ Comput.\ Math., 49, 2023, no. 2,\ Paper No.\ 13), a new approach to POD with DQs was developed that is more computationally efficient than the standard DQ POD approach and it also retains the guaranteed pointwise in time error bounds of the standard method. In this work, we extend this new DQ POD approach to the case of second difference quotients (DDQs). Specifically, a new POD method utilizing DDQs and only one snapshot and one DQ is developed and used to prove ROM error bounds for the damped wave equation. This new approach eliminates data redundancy in the standard DDQ POD approach that uses all of the snapshots, DQs, and DDQs. We show that this new DDQ approach also has pointwise in time data error bounds similar to DQ POD and use it to prove pointwise and energy ROM error bounds. We provide numerical results for the POD errors and ROM errors to demonstrate the theoretical results. We also explore an application of POD to simulating ROMs past the training interval for collecting the snapshot data for the standard POD approach and the DDQ POD method.
	\end{abstract}
	
	
\section{Introduction}

Simulation of high dimensional systems, often times based on partial differential equations (PDEs), is of great importance to applied computational research as well as industry related problems on fluids, heat, and control theory. Often times it is possible to compute the solutions to these high dimensional problems but this requires long computation times. Model order reduction allows for these high dimensional systems to be represented by a low order approximation while still retaining the desired accuracy. Reduced order models (ROMs) can be formed in various ways, but a common technique is proper orthogonal decomposition (POD). POD ROMs are useful for forming accurate low order systems efficiently. Example applications of model order reduction can be found in, e.g., \cite{Batten, Reyes, Balajewicz, Lee, Alla, Bergmann1, Gras1, Bergmann2, Gras2, Georgiou, Amsallem1, Kunisch, Rehm, Sun, Sun2021, Ma2020}. 

The wide appeal of POD in applied research has led numerous researchers to study the numerical analysis aspects of POD ROMs; see, e.g., \cite{Herkt, Amsallem2, Sarahs, Koc1, Koc2, Kostova2018, kunisch2, garcia1, garcia2, garcia3, garcia4, Xie2018, Alla, Kunisch, Singler, Iliescu}. Due to the widespread use of POD in application problems, understanding the sizes of errors involved in using the ROM is extremely important. When simulating PDEs, two types of errors are introduced: spatial discretization error and time discretization error. Research on the two PDE discretization errors is numerous. The ROM introduces a new error: the ROM discretization error. The introduction of the ROM changes how the time discretization error behaves as well so it is common for POD based papers to consider only the time and ROM discretization errors and leave the spatial discretization error to be studied using current methods.

The POD ROM discretization error depends on the method used to construct the POD modes from the data. Koc et al.~\cite{Koc1} recently proved that the standard approach to POD using only the data snapshots does not have pointwise error bounds for the data, while including the difference quotients (DQs) in the snapshots yields pointwise error bounds. Researchers have recently used pointwise error bounds for POD with DQs to analyze DQ POD ROM errors for parabolic PDEs; see, e.g., \cite{Koc1, garcia1, garcia2, garcia2, Koc2, Sarahs}.

In \cite{Herkt, Amsallem2}, researchers derived sum of squares error bounds for the wave equation using all of the snapshots, DQs, and DDQs and the Newmark scheme for the time iteration. Here, we extend DQ POD work in \cite{Sarahs} to DDQs and remove the redundancy in that data set by developing a method using only one snapshot, one DQ, and all of the DDQs. We then develop pointwise data error bounds similar to the bounds shown in \cite{Sarahs}. We also prove pointwise in time energy and \(L^2\) error bounds for a damped wave equation using a simple time stepping scheme. 

We begin with a brief review of the DQ POD method in \cite{Sarahs} that uses only one regular snapshot and all DQs. We present our new extension to DDQs, new results on the POD data approximation errors, and the POD ROM error analysis for a damped wave equation with the new DDQ POD approach. We present numerical results involving the POD data errors and ROM energy and pointwise errors. We also explore changing the training interval for collecting the data to create the POD modes. We compute the final time errors between the finite element solution and the POD ROM solution when the final time lies outside of the training interval. This is of interest as a primary application of a POD ROM is to simulate an equation into the future. 

\section{Proper Orthogonal Decomposition}

Proper Orthogonal Decomposition is a method of reducing the amount of information required to represent a data set. We aim to find a basis that can approximate the data in a way that minimizes a certain error. This optimization forms the basic POD problem. How do we optimally find a basis to minimize the error between our new POD approximate data and the actual data? The core difference between the different POD approaches is how we choose the error we minimize. This choice is often guided by the structure of the problem we aim to use POD on. 

In this section, we briefly review POD following the exposition in \cite{Sarahs}. This section introduces two methods of POD. The first method is the standard approach to POD using only the data snapshots, and the second method is from \cite{Sarahs} and uses one snapshot and all of the difference quotients for the data. The standard POD approach is well known, and details can be found in many references, such as \cite{Koc1, kunisch2}. The recent difference quotient approach from \cite{Sarahs} has been explored and generalized further in \cite{garcia1, garcia2, garcia3}. 

\subsection{Notation}
First, we establish some general notation and define a few key objects. Throughout this work, \(X\) and \(Y\) are separable Hilbert Spaces; for the specific PDE we consider, we often take these spaces to be either \(L^2(\Omega)\) or \(H^1_0(\Omega)\) where \(\Omega\) is a spatial domain. The Hilbert space \(X\) is called the POD space. Let \(M\) be a positive integer. Then the weighted inner product on the space \(S := \mathbb{R}^M_\Gamma\) is defined by \begin{equation*}
    (g,h)_S = h^*\Gamma g = \sum_{j = 1}^M \gamma_j g^j h^j
\end{equation*}
where \(g,h \in S\), \(\Gamma = \mathrm{diag}(\gamma_1, \gamma_2, \ldots, \gamma_M)\), and each \(\gamma_j\) is positive for \(j = 1, \ldots, M\). The constants \(\{\gamma_j\}^M_{j=1}\) are often chosen to approximate a time integral or constant multiple of a time integral. 

For POD reduced order modeling, we consider data sets formed by a finite element (FE) solution of a time dependent PDE. For the data we consider a training interval of \([0,T_t]\) and a testing interval of \([0,T]\) where \(T \geq T_t > 0\). The training interval is the interval of time we take snapshots from the FE solution and the testing interval is the time interval on which we simulate the POD ROM. For the training data we have the FE solution data at times \(t_n = (n-1)\Delta t\) for \(n = 1, \ldots, N\), where \(N > 0\) and \(\Delta t = \frac{T_t}{N-1}\). Unless otherwise stated, \(T_t = T\) and we work with the testing and training intervals being the same. 

An important part of POD is the use of projections. Let \(Z\) be a normed space and let \(Z_r \subset Z\) be a subspace. The bounded linear operator \(\Pi: Z \to Z\) is a projection onto \(Z_r\) if \(\Pi^2 = \Pi\) and \(\mathrm{range} (\Pi) = Z_r\). We then have \(\Pi z = z\) for \(z \in Z_r\). The projections in this work are not required to be orthogonal unless stated otherwise. 

For convenience in presenting POD results when using the DQs and DDQs, we introduce the following notation. For \(\{z^j\}_{j=1}^M \subset Z\), define
\begin{equation*}
    \p z^j = \pp z^j = \frac{z^{j+1} - z^j}{\Delta t},  \qquad  \pn z^j = \frac{z^{j} - z^{j-1}}{\Delta t},
\end{equation*}
and
\begin{equation*}
    \p\p z^j = \pp \pn z^j = \pn\pp z^j = \frac{z^{j+1} - 2z^j + z^{j-1}}{\Delta t^2}.
\end{equation*}
We use the \(\p\) notation for the forward difference as this is the DQ form we use and appears most often in results. In results where the backwards difference appears we use \(\pn\) for the operator. Finally, the \(\p\p\) notation is used for convenience and visual clarity and should not be interpreted as \(\pp\pp\). 

\subsection{Standard POD}\label{sec:2_StPOD}

\qquad We begin by introducing the standard POD problem and operator. 
Let \(W = \{w^j\}^N_{j=1} \subset X\) be the POD data, called the snapshots, for some integer \(N > 0\). Given \(r>0\), the standard POD problem is to find an orthonormal basis \(\{\varphi_k\}^N_{k=1} \subset X\), called the POD basis, minimizing the data approximation error 
\begin{equation} \label{eq:2_StPODErrorEq}
    E_r = \sum_{j=1}^N \gamma_j \|w^j - \Pi^X_r w^j \|^2_X,
\end{equation}
where \(\Pi^X_r:X\to X\) is the orthogonal projection onto \(X_r = \Span\{\varphi_k\}^r_{k=1}\) defined by
\begin{equation} \label{eq:2_L2Proj}
    \Pi^X_r x = \sum_{k=1}^r (x, \varphi_k)_X \varphi_k.
\end{equation}
The POD operator that provides the solution to this problem is \(K:S \to X\)
\begin{equation} \label{eq:2_StPODop}
    Kf = \sum_{j=1}^N \gamma_j f^jw^j, \quad f = [f^1, f^2, \cdots, f^N]^T.
\end{equation}
The operator \( K \)  is called the standard POD operator. It is compact and has a singular value decomposition with \(\{\lambda_k^{1/2}, f_k, \varphi_k\} \subset \real \times S\times X\), where \(\{\lambda_k^{1/2}\}\) are the singular values and \(\{f_k\}\) and \(\{\varphi_k\}\) are the orthonormal singular vectors. Furthermore, we call \(\{\varphi_k\}\) the POD modes of the data and \(\{\lambda_k^{1/2}\}\) the POD singular values. 

We know that the the POD modes give the best low rank approximation to the data, and the standard data error formula is  
\begin{equation} \label{eq:2_StPODErrorFormula}
    E_r = \sum_{j=1}^N \gamma_j\|w^j-\Pi_r^X w^j\|^2_X = \sum_{k=r+1}^s \lambda_k 
\end{equation}
where \(s\) is the number of positive POD singular values. 

The \( \{ \gamma_j \} \) are positive weights that must be specified. They can be selected so that the data approximation error approximates a time integral. It is useful to leave the weights in a general form here so that they may be varied for later POD methods. 

The following lemma states error formulas for norms and projections other than the standard POD norm and projection.
\begin{lem}[Standard POD Extended Data Error Formulas, {\cite[Lemma 1]{Sarahs}}]\label{lem:2_StPODExtendedErrors}
Let \; \(W = \{w^j\}_{j=1}^N\) be the snapshots, \(X_r = \Span\{\varphi_k\}_{k=1}^r\), and \(\Pi_r^X:X\to X\) be the orthonormal projection onto \(X_r\). Let \(s\) be the number of positive POD singular values for \(K\) defined in Equation \eqref{eq:2_StPODop}. If \(Y\) is a Hilbert space with \(W \subset Y\) then 
\begin{equation} \label{eq:2_StPODErrorY}
    \sum_{j=1}^N \gamma_j\|w^j-\Pi_r^X w^j\|^2_Y = \sum_{k=r+1}^s \lambda_k \|\varphi_k\|^2_Y.
\end{equation}
In addition if \(\pi_r:Y\to Y\) is a bounded linear projection onto \(X_r\) then
\begin{equation} \label{eq:2_StPODErrorYpi_r}
    \sum_{j=1}^N \gamma_j\|w^j-\pi_r w^j\|^2_Y = \sum_{k=r+1}^s \lambda_k \|\varphi_k - \pi_r \varphi_k\|^2_Y.
\end{equation}
\end{lem}
This standard method for POD does not have general pointwise error bounds, shown in \cite{Koc1}.

\subsection{POD with 1st Difference Quotients} \label{sec:2_DQPOD}

The following method for POD was proposed in \cite{Sarahs}, extending on the work done in \cite{Koc1}, and has general pointwise error bounds. In this approach, we consider the first data snapshot and all of the difference quotients for the data, defined as the forward difference: 
\begin{equation} \label{eq:2_DQDef}
    \p u^j = \frac{u^{j+1}-u^j}{\Delta t}.
\end{equation}
Then for the data \(U = \{u^j\}_{j=1}^N\), the error we aim to minimize is
\begin{equation} \label{eq:2_DQPODErrorEq}
    E_r^{\mathrm{DQ1}} = \|u^1 - \Pi_r^X u^1\|^2_X + \sum_{j=1}^{N-1} \Delta t\|\p u^j - \Pi_r^X \p u^j\|^2_X.
\end{equation}
This error can be found with the POD operator:
\begin{equation} \label{eq:2_DQPODop}
    K_1f = f^1u^1 + \sum_{j=1}^{N-1} \Delta t f^{j+1}\p u^j.
\end{equation}
Here \(K_1f = Kf\) where \(w^1 = u^1\) and \(w^{j+1} = \p u^j\) for \(j = 1, \ldots, N-1\) with \(\gamma_1 = 1\) and \(\gamma_j = \Delta t\) for \(j = 2,\ldots, N\). With \(\{\lambda_j^{\mathrm{DQ1}}\}_{j=1}^N\) as the POD eigenvalues and \(\{\varphi_k\}_{k=1}^r\) as the POD modes for the data, the following lemma provides error formulas for the data approximation. 

\begin{lem}[DQ1 POD Extended Data Error Formulas, {\cite[Lemma 5]{Sarahs}}] \label{lem:2_DQPODErrorFormulas}
Let \(U = \{u^j\}_{j=1}^N\) be the snapshots, \(X_r = \Span\{\varphi_k\}_{k=1}^r\), and \(\Pi_r^X:X\to X\) be the orthonormal projection onto \(X_r\). Let \(s\) be the number of positive POD singular values for \(K_1\) defined in Equation \eqref{eq:2_DQPODop}. Then
\begin{equation} \label{eq:2_DQPODError}
    \|u^1 - \Pi_r^X u^1\|^2_X + \sum_{j=1}^{N-1} \Delta t\|\p u^j - \Pi_r^X \p u^j\|^2_X = \sum_{k=r+1}^s \lambda_k^{\mathrm{DQ1}}. 
\end{equation}
If \(Y\) is a Hilbert space with \(W \subset Y\) then 
\begin{equation} \label{eq:2_DQPODErrorY}
    \|u^1 - \Pi_r^X u^1\|^2_Y + \sum_{j=1}^{N-1} \Delta t\|\p u^j - \Pi_r^X \p u^j\|^2_Y = \sum_{k=r+1}^s \lambda_k^{\mathrm{DQ1}}\|\varphi_k\|^2_Y. 
\end{equation}
In addition if \(\pi_r:Y\to Y\) is a bounded linear projection onto \(X_r\) then
\begin{equation} \label{eq:2_DQPODErrorYpi_r}
    \|u^1 - \pi_r u^1\|^2_Y + \sum_{j=1}^{N-1} \Delta t\|\p u^j - \pi_r \p u^j\|^2_Y = \sum_{k=r+1}^s \lambda_k^{\mathrm{DQ1}}\|\varphi_k - \pi_r \varphi_k\|^2_Y.
\end{equation}
\end{lem}
The following lemma was used in \cite{Sarahs} to prove pointwise error formulas and will be used to prove new error bounds in Section \ref{sec:3_DDQ}.
\begin{lem}[General Pointwise Norm Bounds, {\cite[Lemma 6]{Sarahs}}] \label{lem:2_DQPointwiseErrorBounds} 
Let \(T > 0\), \(Z\) be a normed space, \(\{z^j\}_{j=1}^N \subset Z\), and \(\Delta t = T/(N-1)\). Then 
\begin{equation} \label{eq:2_z_DQpwErrorBound}
    \max_{1\leq j \leq N}\|z^j\|^2_Z \leq C_1\left(\|z^1\|^2_Z + \sum_{\ell=1}^{N-1} \Delta t \|\p z^\ell\|^2_Z\right)
\end{equation}
where \(C_1 = 2\max\{T, 1\}\).
\end{lem}

The difference quotient approach to POD then has the following pointwise error bounds. 

\begin{theorem}[Pointwise Data Error Bounds for \(K_1\), {\cite[Theorem 7]{Sarahs}}] \label{thm:2_dataDQpwErrorBounds}
    Let \(U = \{u^j\}_{j=1}^N\) be the snapshots, \(X_r = \Span\{\varphi_k\}_{k=1}^r\) and \(\Pi^X_r : X\to X\) be the orthogonal projection onto \(X_r\). Let \(s\) be the number of positive POD eigenvalues for \(K_1\). Then 
    \begin{equation} \label{eq:2_dataDQpwError}
        \max_{1\leq j \leq N}\|u^j-\Pi^X_r u^j\|^2_X \leq C\left(\sum_{k=r+1}^s \lambda^{\mathrm{DQ1}}_k\right).
    \end{equation}
    If \(Y\) is a Hilbert space with \(U \subset Y\) then 
    \begin{equation} \label{eq:2_dataDQpwErrorY}
        \max_{1\leq j \leq N}\|u^j-\Pi^X_r u^j\|^2_Y \leq C\left(\sum_{k=r+1}^s \lambda^{\mathrm{DQ1}}_k\|\varphi_k\|^2_Y\right),
    \end{equation}
    and in addition if \(\pi_r:Y\to Y\) is a bounded linear projection onto \(X_r\) then 
    \begin{equation} \label{eq:2_dataDQpwErrorYpi_r}
        \max_{1\leq j \leq N}\|u^j-\pi_r u^j\|^2_Y \leq C\left(\sum_{k=r+1}^s \lambda^{\mathrm{DQ1}}_k\|\varphi_k-\pi_r \varphi_k\|^2_Y\right),
    \end{equation}
    where \(C = 2\max\{T, 1\}\).
\end{theorem}
The following corollary from \cite{Sarahs} states results for weighted sums of the snapshot data errors. 
\begin{corollary}[Weighted Sum Data Error Bounds, {\cite[Corollary 8]{Sarahs}}]\label{cor:2_dataSumDQpwErrorBounds}
    Let \(U = \{u^j\}_{j=1}^N\) be the snapshots, \(X_r = \Span\{\varphi_k\}_{k=1}^r\), and \(\Pi^X_r : X\to X\) be the orthogonal projection onto \(X_r\). Let \(s\) be the number of positive POD eigenvalues for \(K_1\). Then 
    \begin{equation} \label{eq:2_dataSumDQpwError}
        \sum_{j=1}^N \Delta t\|u^j-\Pi^X_r u^j\|^2_X \leq C\left(\sum_{k=r+1}^s \lambda^{\mathrm{DQ1}}_k\right).
    \end{equation}
    If \(Y\) is a Hilbert space with \(U \subset Y\) then 
    \begin{equation} \label{eq:2_dataSumDQpwErrorY}
        \sum_{j=1}^N \Delta t\|u^j-\Pi^X_r u^j\|^2_Y \leq C\left(\sum_{k=r+1}^s \lambda^{\mathrm{DQ1}}_k\|\varphi_k\|^2_Y\right).
    \end{equation}
    If in addition if \(\pi_r:Y\to Y\) is a bounded linear projection onto \(X_r\) then 
    \begin{equation} \label{eq:2_dataSumDQpwErrorYpi_r}
        \sum_{j=1}^N \Delta t\|u^j-\pi_r u^j\|^2_Y \leq C\left(\sum_{k=r+1}^s \lambda^{\mathrm{DQ1}}_k\|\varphi_k-\pi_r \varphi_k\|^2_Y\right)
    \end{equation}
    where \(C = 4\max\{T^2, T\}\).
\end{corollary}
These results were applied to the heat equation and specifically Lemma \ref{lem:2_DQPODErrorFormulas}, Theorem \ref{thm:2_dataDQpwErrorBounds}, and Corollary \ref{cor:2_dataSumDQpwErrorBounds} were used to prove pointwise error bounds for the POD reduced order model of the heat equation in \cite{Sarahs}.

\section{A New Method for POD with Second Difference Quotients}
\label{sec:3_DDQ}

In this section, we propose a new POD method using 2nd difference quotients and prove corresponding results on the pointwise data errors. This method extends the approach with 1st difference quotients in Section \ref{sec:2_DQPOD} and on work done in \cite{Amsallem2, Herkt} using all of the snapshots, 1st difference quotients, and 2nd difference quotients. POD with 1st difference quotients allows for ROM error bounds to be proven for heat equation and other 1st order in time PDE problems; see, e.g., \cite{garcia1, garcia2, Koc2}. Here, we utilize the 2nd difference quotients to analyze a 2nd order in time PDE problem.

\subsection{DDQ POD Approach} \label{sec:3_DDQPODmethod}

For the DDQ POD method, we include one snapshot, one 1st difference quotient, and all of the 2nd difference quotients. We use the second difference quotient \begin{equation} \label{eq:3_DDQDef}
    \p\p u^j = \frac{u^{j+1}-2u^j + u^{j-1}}{\Delta t^2}.
\end{equation} 
This means that for the data \(U = \{u^j\}_{j=1}^N\),  we aim to minimize the following error 
\begin{equation} \label{eq:3_DDQErrorEq}
    E_r^{\mathrm{DDQ}} = \|u^1 - \Pi_r^X u^1\|^2_X + \|\p u^1 - \Pi_r^X \p u^1\|^2_X + \sum_{j=2}^{N-1} \Delta t\|\p\p u^j - \Pi_r^X \p\p u^j\|^2_X.
\end{equation}
This error function has a similar structure to the method in Section \ref{sec:2_DQPOD} where now we select all the 2nd difference quotients to have a weight of \( \Delta t\) and the first snapshot and first difference quotient are weighted by 1.

The POD operator corresponding to this error is 
\begin{equation} \label{eq:3_DDQPODop}
    K_2f = f^1u^1 + f^2\p u^1 + \sum_{j=2}^{N-1} \Delta t f^{j+1}\p\p u^j.
\end{equation}
Here \(K_2f = Kf\) where \(w^1 = u^1\), \(w^2 = \p u^1\), and \(w^{j+1} = \p\p u^j\) for \(j = 2, \ldots, N-1\) with \(\gamma_1 = \gamma_2 = 1\) and \(\gamma_j = \Delta t\) for \(j = 3,\ldots, N\). 
\begin{lem}[Linear Independence of 2nd Difference Quotient Data Set]\label{lem:3_LinIndDDQdata}

If \(\{u^i\}_{i=1}^N\) is linearly independent, then \(\{w^i\}_{i=1}^N\) given by \(w^1 = u^1\), \(w^2 = \p u^1\) and \(w^{i+1} = \p\p u^i\) for \(i = 2,\ldots,N-1\) is linearly independent.
\end{lem}
The proof is similar to the proof of \cite[Lemma 4]{Sarahs} and is omitted.

With \(\{\lambda_j^{\mathrm{DDQ}}\}_{j=1}^N\) as the POD eigenvalues and \(\{\varphi_k\}_{k=1}^r\) as the POD modes for the data,  Lemma \ref{lem:3_DDQPODErrorFormulas} provides error formulas for the data approximation.

\begin{lem}[DDQ POD Extended Data Error Formulas]\label{lem:3_DDQPODErrorFormulas}

Let \(U = \{u^j\}_{j=1}^N\) be the snapshots, \(X_r = \Span\{\varphi_k\}_{k=1}^r\), and \(\Pi_r^X:X\to X\) be the orthonormal projection onto \(X_r\). Let \(s\) be the number of positive POD singular values for \(K_2\) defined in Equation \eqref{eq:2_DQPODop}. Then
\begin{equation} \label{eq:3_DQPODError}
    \|u^1 - \Pi_r^X u^1\|^2_X + \|\p u^1 - \Pi_r^X \p u^1\|^2_X + \sum_{j=2}^{N-1} \Delta t\|\p\p u^j - \Pi_r^X \p\p u^j\|^2_X = \sum_{k=r+1}^s \lambda_k^{\mathrm{DDQ}}.
\end{equation}
If \(Y\) is a Hilbert space with \(U \subset Y\), then 
\begin{equation} \label{eq:3_DQPODErrorY}
    \|u^1 - \Pi_r^X u^1\|^2_Y + \|\p u^1 - \Pi_r^X \p u^1\|^2_Y + \sum_{j=2}^{N-1} \Delta t\|\p\p u^j - \Pi_r^X \p\p u^j\|^2_Y = \sum_{k=r+1}^s \lambda_k^{\mathrm{DDQ}}\|\varphi_k\|^2_Y.
\end{equation}
In addition if \(\pi_r:Y\to Y\) is a bounded linear projection onto \(X_r\) then
\begin{equation} \label{eq:3_DQPODErrorYpi_r}
    \|u^1 - \pi_r u^1\|^2_Y + \|\p u^1 - \pi_r \p u^1\|^2_Y + \sum_{j=2}^{N-1} \Delta t\|\p\p u^j - \pi_r \p\p u^j\|^2_Y = \sum_{k=r+1}^s \lambda_k^{\mathrm{DDQ}}\|\varphi_k - \pi_r \varphi_k\|^2_Y. 
\end{equation}
\end{lem}
\begin{proof}
    This follows from Equation \eqref{eq:2_StPODErrorFormula} and Lemma \ref{lem:2_StPODExtendedErrors} where \(\{\lambda_j^{\mathrm{DDQ}}\}_{j=1}^N\) are taken as the POD eigenvalues for the POD operator in Equation \eqref{eq:3_DDQPODop}.
\end{proof}
Lemma \ref{lem:3_DDQPODErrorFormulas} will be verified using data from a damped wave equation in Section \ref{sec:5_PODdataComps}.

\subsection{Pointwise Error Bounds}\label{sec:3_PWErrorBounds}

In this section, we prove pointwise error bounds for the data when using DDQ POD. We extend the proof ideas used in \cite{Sarahs} for DQ POD and develop general error formulas which will be used again in Section \ref{sec:4_All} to prove ROM error bounds. 


Lemma \ref{lem:3_zByDDQs} is important for proving Lemma \ref{lem:3_zDDQErrorBounds} which will be the main result for proving the pointwise error bounds for the data and for the ROM.

\begin{lem}[Representing \(z^n\) with 2nd Difference Quotients] \label{lem:3_zByDDQs} 
Let \(\Delta t > 0\) and \(\{z^n\}_{n=1}^N \subset Z\) where \(Z\) is a vector space.
Then
\begin{equation} \label{eq:3_pzDDQEq}
    \p z^n = \p z^1 + \Delta t\sum_{i=2}^{n} \p\p z^i,\quad \text{ for }  n = 2, \ldots, N,
\end{equation}
\begin{equation} \label{eq:3_zDDQEq}
    z^n = z^1 + (n-1)\Delta t\p z^1 + \Delta t^2\sum_{i=2}^{n-1} (n-i)\p\p z^i,  \quad \text{ for }  n = 3, \ldots, N.
\end{equation}
\end{lem}
\begin{proof}
    First, notice that \begin{equation*}
        \p\p z^i = \frac{z^{i+1} - 2z^i + z^{i-1}}{\Delta t^2} = \frac{\p z^{i} - \p z^{i-1}}{\Delta t}.
    \end{equation*}
    Then, \begin{equation*}
        \p z^1 + \Delta t \sum_{i=2}^{n} \p\p z^i = \p z^1 + \Delta t \sum_{i=2}^{n} \frac{\p z^{i} - \p z^{i-1}}{\Delta t} = \p z^1 + \sum_{i=2}^{n} \p z^{i} - \p z^{i-1} = \p z^1 + \p z^n - \p z^1 = \p z^n.
    \end{equation*}
    This sum clearly telescopes and yields 
    so \eqref{eq:3_pzDDQEq} is proven. 

    To prove \eqref{eq:3_zDDQEq}, first consider \eqref{eq:3_pzDDQEq} with \(n = j\), sum over \(j = 2, \ldots, n-1\), and multiply by \(\Delta t\) to get
     \begin{equation*}
       \Delta t\sum_{j=2}^{n-1} \left(\p z^j - \p z^1\right) = \Delta t^2 \sum_{j=2}^{n-1}\sum_{i=2}^{j} \p\p z^i.
    \end{equation*}
    Add and subtract \(\Delta t \p z^1\) on the left hand side so that the sum goes from \(j = 1, \ldots, n-1\): 
    \begin{equation*}
       \Delta t\sum_{j=2}^{n-1} \left(\p z^j - \p z^1\right) + \Delta t\p z^1 - \Delta t\p z^1 = \Delta t^2 \sum_{j=2}^{n-1}\sum_{i=2}^{j} \p\p z^i
    \end{equation*}
    \begin{equation*}
        \Longrightarrow \qquad \Delta t\sum_{j=1}^{n-1} \left(\p z^j - \p z^1\right)= \Delta t^2 \sum_{j=2}^{n-1}\sum_{i=2}^{j} \p\p z^i.
    \end{equation*}
    Evaluating the left hand side and rearranging, we have
    \begin{equation*}
       z^{n} = z^1 + (n-1)\Delta t\p z^1 + \Delta t^2 \sum_{j=2}^{n-1}\sum_{i=2}^{j} \p\p z^i.
    \end{equation*}
    Next, we swap the order of the summations with \( 2 \leq i \leq n-1\) and \( i \leq j \leq n - 1\), yielding
    \begin{align*}
       z^{n} = z^1 + (n-1)\Delta t\p z^1 + \Delta t^2 \sum_{i=2}^{n-1}\sum_{j=i}^{n-1} \p\p z^i = z^1 + (n-1)\Delta t\p z^1 + \Delta t^2 \sum_{i=2}^{n-1} (n - i) \p\p z^i.
    \end{align*}
\end{proof}

\begin{lem}[Pointwise Error Bounds for a General Function] \label{lem:3_zDDQErrorBounds} 

Let \(T > 0\), \(N > 0\), \(Z\) be a normed space, \(\{z^j\}_{j=1}^N \subset Z\), \(\Delta t = T/(N-1)\), and define the backwards average by \begin{equation*}
    \overline{z}^j = \frac{z^j + z^{j-1}}{2}.
\end{equation*} Then 

\begin{equation} \label{eq:3_zDDQpwBound}
   \max_{1\leq j \leq N}\|z^j\|_Z^2 \leq C_2\left(\|z^1\|_Z^2 + \|\p z^1\|_Z^2 + \sum_{i=2}^{N-1} \Delta t\|\p\p z^i\|_Z^2\right),
\end{equation}
\begin{equation} \label{eq:3_zavgDDQpwBound}
   \max_{2 \leq j \leq N}\|\overline{z}^j\|_Z^2 \leq C_2\left(\|z^1\|_Z^2 + \|\p z^1\|_Z^2 + \sum_{i=2}^{N-1} \Delta t\|\p\p z^i\|_Z^2\right),
\end{equation}
\begin{equation} \label{eq:3_pzDDQpwBound}
    \max_{1 \leq j \leq N-1} \|\p z^j\|_Z^2 \leq C_3\left(\|\p z^1\|_Z^2 + \sum_{i=2}^{N-1} \Delta t\|\p\p z^i\|_Z^2\right),
\end{equation}
\begin{equation} \label{eq:3_pnzDDQpwBound}
    \max_{2 \leq j \leq N} \|\pn z^j\|_Z^2 \leq C_3\left(\|\p z^1\|_Z^2 + \sum_{i=2}^{N-1} \Delta t\|\p\p z^i\|_Z^2\right),
\end{equation}
\begin{equation} \label{eq:3_pavgzDDQpwBound}
    \max_{2 \leq j \leq N-1} \|\p \overline{z}^j\|_Z^2 \leq C_3\left(\|\p z^1\|_Z^2 + \sum_{i=2}^{N-1} \Delta t\|\p\p z^i\|_Z^2\right)
\end{equation}
where \(C_2 = 3\max\{T^3, 1\}\) and \(C_3 = 2\max\{T, 1\}\).
\end{lem}
\begin{proof}
To prove \eqref{eq:3_zDDQpwBound}, use \eqref{eq:3_zDDQEq} from Lemma \ref{lem:3_zByDDQs}, take norms and use \((a + b+ c)^2 \leq 3(a^2 + b^2 + c^2)\) to get
\begin{align*}
   \|z^n\|_Z^2 \leq 3\left(\|z^1\|_Z^2 + (n-1)^2\Delta t^2\|\p z^1\|_Z^2 + (n-1)^3\Delta t^3 \sum_{i=2}^{n-1} \Delta t\|\p\p z^i\|_Z^2\right).
\end{align*}
We also have \(T \geq T_n = (n-1)\Delta t\), so
\begin{align*}
   \|z^n\|^2 \leq 3\left(\|z^1\|^2 + T^2\|\p z^1\|^2 + T^3 \sum_{i=2}^{n-1} \Delta t\|\p\p z^i\|^2\right).
\end{align*}
With \(C_2 = 3\max\{T^3, 1\}\),
\begin{align*}
   \|z^n\|_Z^2 \leq C_2\left(\|z^1\|_Z^2 + \|\p z^1\|_Z^2 + \sum_{i=2}^{n-1} \Delta t\|\p\p z^i\|_Z^2\right).
\end{align*} 
Taking the maximum over all \(n\) proves Equation \eqref{eq:3_zDDQpwBound}.
Since \((a+b)^2 \leq 2(a^2+b^2)\), we also have 
\begin{equation*}
    \|\overline{z}^n\|_Z^2 \leq \frac{1}{2}\left(\|z^n\|^2_Z + \|z^{n-1}\|^2_Z\right) \leq \max_{1\leq j \leq N} \|z^j\|^2_Z,
\end{equation*}
which proves \eqref{eq:3_zavgDDQpwBound}.

To prove \eqref{eq:3_pzDDQpwBound}, use \eqref{eq:3_pzDDQEq}, take norms, and use the triangle inequality to yield
\begin{align*}
    \|\p z^n\|_Z \leq \|\p z^1\|_Z + \sum_{i=2}^{n} \Delta t\|\p\p z^i\|_Z.
\end{align*}
Using Cauchy-Schwarz on the sum term, we have
\begin{align*}
    \|\p z^n\|_Z \leq \|\p z^1\|_Z + \left(\sum_{i=2}^n \Delta t\right)^{1/2}\left(\sum_{i=2}^{n} \Delta t\|\p\p z^i\|_Z^2\right)^{1/2}.
\end{align*}
Again using \((a+b)^2 \leq 2(a^2+b^2)\), 
\begin{align*}
    \|\p z^n\|_Z^2 \leq 2\left(\|\p z^1\|_Z^2 + \left(\sum_{i=2}^n \Delta t\right)\left(\sum_{i=2}^{n} \Delta t\|\p\p z^i\|_Z^2\right)\right)
\end{align*}
Once again, since \(T \geq T_n\), we have
\begin{align*}
    \|\p z^n\|^2 \leq 2\left(\|\p z^1\|^2 + T\sum_{i=2}^{n} \Delta t\|\p\p z^i\|^2\right).
\end{align*}
Finally letting \(C_3 = 2\max\{T,1\}\) and taking the maximum over all \(n\) yields Equation \eqref{eq:3_pzDDQpwBound}.
The last two results can be easily shown since \begin{equation*}
    \p z^n = \pn z^{n+1}
\end{equation*}
and 
\begin{equation*}
    \|\p \overline{z}^n\|_Z^2 \leq \frac{1}{2}\left(\|\p z^n\|^2_Z + \|\p z^{n-1}\|^2_Z\right) \leq \max_{1\leq j \leq N-1} \|\p z^j\|^2_Z.
\end{equation*}
\end{proof}
With Lemma \ref{lem:3_zDDQErrorBounds} we can prove pointwise error bounds for the DDQ approach for POD. 
\begin{theorem} \label{thm:3_dataDDQErrorBounds}
    Let \(U = \{u^j\}_{j=1}^N\) be the snapshots, \(X_r = \Span\{\varphi_k\}_{k=1}^r\), and \(\Pi^X_r : X\to X\) be the orthogonal projection onto \(X_r\). Let \(s\) be the number of positive POD eigenvalues for \(K_2\). Then 
    \begin{equation} \label{eq:3_dataDDQErrorBound}
        \max_{1\leq j \leq N}\|u^j-\Pi^X_r u^j\|^2_X \leq C\left(\sum_{k=r+1}^s \lambda^{\mathrm{DDQ}}_k\right).
    \end{equation}
    If \(Y\) is a Hilbert space with \(U \subset Y\) then 
    \begin{equation} \label{eq:3_dataDDQErrorBoundY}
        \max_{1\leq j \leq N}\|u^j-\Pi^X_r u^j\|^2_Y \leq C\left(\sum_{k=r+1}^s \lambda^{\mathrm{DDQ}}_k\|\varphi_k\|^2_Y\right),
    \end{equation}
    and in addition if \(\pi_r:Y\to Y\) is a bounded linear projection onto \(X_r\) then 
    \begin{equation} \label{eq:3_dataDDQErrorBoundYpi_r}
        \max_{1\leq j \leq N}\|u^j-\pi_r u^j\|^2_Y \leq C\left(\sum_{k=r+1}^s \lambda^{\mathrm{DDQ}}_k\|\varphi_k-\pi_r \varphi_k\|^2_Y\right).
    \end{equation}
    where \(C = 3\max\{T^3, 1\}\)
\end{theorem}
\begin{proof}
    Using Lemma \ref{lem:3_zDDQErrorBounds} with \(z^j = u^j-\Pi^X_r u^j\) and \(Z = X\), we have 
    \begin{align*}
        \max_{1\leq j \leq N}\|u^j-\Pi^X_r u^j\|^2_X \leq C_2 & \Biggl( \|u^1 - \Pi_r^X u^1\|^2_X + \|\p u^1 - \Pi_r^X \p u^1\|^2_X \\ 
        & \qquad  +\sum_{j=2}^{N-1} \Delta t\|\p\p u^j - \Pi_r^X \p\p u^j\|^2_X \Biggl).
    \end{align*}
    Applying Lemma \ref{lem:3_DDQPODErrorFormulas}, we have 
    \begin{equation*}
        \max_{1\leq j \leq N}\|u^j-\Pi^X_r u^j\|^2_X \leq C_2\left(\sum_{k=r+1}^s \lambda^{\mathrm{DDQ}}_k\right),
    \end{equation*}
    and renaming \(C = C_2\) proves Equation \eqref{eq:3_dataDDQErrorBound}.
    Following the same process with \(z^j = u^j-\Pi^X_r u^j\) and \(Z = Y\) and with \(z^j = u^j-\pi_r u^j\) and \(Z = Y\) proves Equations \eqref{eq:3_dataDDQErrorBoundY} and \eqref{eq:3_dataDDQErrorBoundYpi_r} respectively. 
\end{proof}
Corollary \ref{cor:3_dataSumDDQErrorBounds} corresponds to bounding a discrete time integral of the data error.
\begin{corollary} \label{cor:3_dataSumDDQErrorBounds}
    Let \(U = \{u^j\}_{j=1}^N\) be the snapshots, \(X_r = \Span\{\varphi_k\}_{k=1}^r\), and \(\Pi^X_r : X\to X\) be the orthogonal projection onto \(X_r\). Let \(s\) be the number of positive POD eigenvalues for \(K_2\). Then 
    \begin{equation} \label{eq:3_dataSumDDQErrorBound}
        \sum_{j=1}^N \Delta t\|u^j-\Pi^X_r u^j\|^2_X \leq C\left(\sum_{k=r+1}^s \lambda^{\mathrm{DDQ}}_k\right).
    \end{equation}
    If \(Y\) is a Hilbert space with \(U \subset Y\) then 
    \begin{equation} \label{eq:33_dataSumDDQErrorBoundY}
        \sum_{j=1}^N \Delta t\|u^j-\Pi^X_r u^j\|^2_Y \leq C\left(\sum_{k=r+1}^s \lambda^{\mathrm{DDQ}}_k\|\varphi_k\|^2_Y\right).
    \end{equation}
    If in addition if \(\pi_r:Y\to Y\) is a bounded linear projection onto \(X_r\) then 
    \begin{equation} \label{eq:3_dataSumDDQErrorBoundYpi_r}
        \sum_{j=1}^N \Delta t\|u^j-\pi_r u^j\|^2_Y \leq C\left(\sum_{k=r+1}^s \lambda^{\mathrm{DDQ}}_k\|\varphi_k-\pi_r \varphi_k\|^2_Y\right).
    \end{equation}
    where \(C = 6\max\{T^4, T\}\).
    
\end{corollary}
\begin{proof}
    The proof of this result is the same as the proof for Corollary 8 in \cite{Sarahs} and is omitted. 
\end{proof}

As with the DQ approach in \cite{Sarahs}, we have pointwise error formulas and no redundancy in the data set.

\section{Reduced Order Model Error Analysis}

\label{sec:4_All}
In this section, we present the chosen PDE problem, the damped wave equation, and the finite elemtn method used for approximating the solution. We also present the POD ROM for the damped wave equation and derive energy and pointwise error bounds for the ROM.  

\subsection{Finite Element Method for the Damped Wave Equation} \label{sec:4_FEWaveEQmethod}
The problem we choose to analyze with the new POD method is the 1-D damped wave equation with zero Dirichlet boundary conditions, 
\begin{equation}\label{eq:4_WaveEq}
\begin{split}
    u_{tt} - c^2 u_{xx} + Du_t - Gu_{txx} = 0, \quad \text{ in } [0,1]\times [0,T] \\
    u(x, 0) = u_0,  \qquad  u_t(x,0) = u_{00},
\end{split}
\end{equation}
with constants \(c > 0\) and \(D,G \in [0, \infty)\). The constant D is the coefficient of viscous damping and G is the coefficient of Kelvin-Voigt damping. It is important to note that in our analysis and computations \(D\) and \(G\) are never both zero.

The structure of the separation of variables solution demonstrates key differences between both types of damping. Denote \(\lambda_k = \pi k\). The general series solution when \(D > 0\) and \(G = 0\) is
\begin{equation} \label{eq:4_WaveEQsolnD}
    u(x,t) = \sum_{k=1}^\infty e^{-\frac{D}{2}t}\left(a_k e^{\xi_k t}+b_ke^{-\xi_k t}\right)\sin(\lambda_k x),  \quad  \xi_k := \sqrt{\frac{D^2}{4} - c^2\lambda_k^2},
\end{equation}
and when \(D = 0\) and \(G > 0\) the solution is 
\begin{equation} \label{eq:4_WaveEQsolnG}
    u(x,t) = \sum_{k=1}^\infty e^{-\frac{G\lambda_k^2}{2}t}\left(c_k e^{\zeta_k t}+d_ke^{-\zeta_k t}\right)\sin(\lambda_k x),  \quad  \zeta_k := \sqrt{\frac{G^2\lambda_k^4}{4} - c^2\lambda_k^2},
\end{equation}
for some constants \(a_k,b_k,c_k,d_k \in \mathbb{C}\) depending on the initial conditions. 

From the first solution with only viscous damping, we can see that as long as \(D < 2c\pi\), all of the terms are oscillatory and each oscillatory mode decays at the same rate of \(D/2\). The second solution shows that the Kelvin-Voigt damping terms are only oscillatory when \(k < \frac{2c}{\pi G}\). When \(k\) is larger, the mode is overdamped and does not oscillate. We also see that the rate of decay for each mode increases with \(k^2\) so only a few of the oscillatory modes contribute meaningfully to the solution in the long term. These differences are explored later in Section \ref{sec:5_all} when comparing POD for the two different types of damping. 

The initial conditions used throughout this work are \begin{equation} \label{eq:4_ICu_0}
    u_0 = \left(e^x + x^2 - \cos(\pi x)\right)\sin(\pi x) + \left(e^{x^2} + x^2 -x\right)\sin(5\pi x),
\end{equation}
\begin{equation}\label{eq:4_ICu_00}
    u_{00} = 0. 
\end{equation}
This initial condition contains complexity that would lead to more high frequency oscillatory modes in the series solution. This means the POD ROM will need more basis functions to be able to represent those oscillations. This leads to more interesting analysis for the errors and in the plots we present in later sections.

\subsubsection{Finite Element Discretization Scheme}
The weak form of this problem is to find \(u \in H^1_0(0,1)\) satisfying\begin{equation}\label{eq:4_WeakFormCont}
    (u_{tt},v)_{L^2} + c^2 (u_{x},v_{x})_{L^2} + D (u_{t},v)_{L^2} + G (u_{tx}, v_{x})_{L^2} = 0
\end{equation}
for all \( v \in H^1_0(0,1).\)
We use finite elements to approximate the solution \(u_h \in V^h\) with the following weak form
\begin{equation}\label{eq:4_WeakFormVh}
    (u_{htt},v_h)_{L^2} + c^2 (u_{hx},v_{hx})_{L^2} + D (u_{ht},v_h)_{L^2} + G (u_{htx}, v_{hx})_{L^2} = 0
\end{equation}
for \(v_h \in V^h \subset H^1_0(0,1)\), where \( V^h \) is the finite element function space. We use linear FE basis functions so \(V^h = \Span\{\phi_i\}_{i=1}^M\) and the time discretization scheme we use is 
\begin{equation}\label{eq:4_WeakFormFEscheme}
    (\p\p u_h^n,v)_{L^2} + c^2 (\widehat{u}^n_{hx},v_{hx})_{L^2} + D (\p \overline{u}^n_h,v_h)_{L^2} + G (\p \overline{u}^n_{hx}, v_{hx})_{L^2} = 0
\end{equation}
for \(n = 2, \ldots, N-1\),
where 
\begin{equation} \label{eq:4_BackwardsAvgandCenteredAvg}
    \overline{u}^n = \frac{u^n + u^{n-1}}{2} \quad \text{and} \quad \widehat{u}^n = \frac{u^{n+1} + 2u^n + u^{n-1}}{4}
\end{equation}
are discrete time averages of the solution. In the undamped case, the centered time average keeps 2nd order accuracy in the iteration seen in \cite{Dupont}. We use the second order centered difference,
\begin{equation} \label{eq:4_1stCenteredDiff}
    \p \overline{u}^n = \frac{u^{n+1} - u^{n-1}}{2\Delta t}, 
\end{equation} for the damping terms. We do not prove that this discrete scheme is second order accurate for the damped case; we leave this to be considered elsewhere. 
\subsubsection{Finite Element Approximations to the Initial Condition} \label{sec:4_FE_ICs}
In this section, we detail our method of obtaining a 2nd order in time accurate set of initial conditions. For our time discretization method, we need \(u^1_h\) and \(u^2_h\) to be given. Obtaining \(u^1_h\) is simple: we use the \(L^2\) projection of \(u_0\). Getting a 2nd order accurate \(u^2_h\) is the difficult part. The method we use is briefly described in \cite{Dupont} and uses the wave equation itself along with a Taylor expansion of \(u\). For completeness, we provide the details of obtaining the two ICs. 

We use the \(L^2\) projection \(P_h\) onto the FE space \(V^h\) for placing the initial conditions into the FE basis. Specifically, for \(u \in L^2(0,1)\), the projection \(P_h u \in V^h\) is found by solving Equation \eqref{eq:4_L2Proj}:
\begin{equation} \label{eq:4_L2Proj}
    (P_h u, \phi_i)_{L^2} = (u, \phi_i)_{L^2} \ \forall \ i =1, \ldots, M.
\end{equation}
We use \(u_0\) to get \(u^1_h = P_h u_0\). We obtain \(u^2_h\) by first finding a 2nd order accurate approximation \( u^2_2 \) to \(u(x, \Delta t) = u^2 \). We then use the \(L^2\) projection onto the FE basis to get \(u^2_h := P_h u^2_2\) as follows. First, consider this rearranged weak form of the problem in Equation \eqref{eq:4_WeakFormCont}.
\begin{equation} \label{eq:4_ICweakform}
    (u_{tt},v)_{L^2} = - c^2 (u_{x},v_{x})_{L^2} - D (u_{t},v)_{L^2} - G (u_{tx}, v_{x})_{L^2}.
\end{equation}
Performing a Taylor expansion in time of \(u(x,t)\) gives
\begin{equation} \label{eq:4_ICTaylorExp}
    u^2 = u(x, \Delta t) = u(x,0) + \Delta t \, u_t(x,0) + \frac{\Delta t^2}{2} u_{tt}(x, 0) + O(\Delta t^3).
\end{equation}
We drop the \(O(\Delta t^3)\) terms and retain 2nd order accuracy in time to get 
\begin{equation} \label{eq:4_IC2orderExp}
    u^2 \approx u^2_2 := u(x,0) + \Delta t u_t(x,0) + \frac{\Delta t^2}{2} u_{tt}(x, 0).
\end{equation}
Since \(u^2_h := P_h u^2_2\), use Equations \eqref{eq:4_L2Proj} and \eqref{eq:4_ICweakform} with \(u = u^2_2\) to obtain \begin{align}
    (u^2_h, \phi_i)_{L^2} &= (u^2_2, \phi_i)_{L^2} \nonumber\\
    &= (u_0,\phi_i)_{L^2} + \Delta t(u_{00},\phi_i)_{L^2} \\ & \: \ \quad\qquad \qquad - \frac{\Delta t^2}{2}\left( c^2(u_{0x},\phi_{ix})_{L^2} + G(u_{00x},\phi_{ix})_{L^2} + D(u_{00},\phi_{i})_{L^2}\right)\nonumber
\end{align}
for all \(i = 1, \ldots, M\). Solving this system yields the second initial condition. We enforce the zero Dirichlet boundary conditions for each case. With this method we obtain a 2nd order accurate set of initial conditions. 

\subsection{Introducing the ROM}\label{sec:4_ROM}
For the error analysis, we analyze a more general PDE problem, namely the damped wave equation in multiple spacial dimensions. 
Let \(\Omega = \real^d\), for \(d\geq 1\), be an open bounded domain with Lipschitz continuous boundary and define \(V = H^1_0(\Omega)\). The space V is a Hilbert space with inner product \((g,h)_{H^1_0} = (\nabla g, \nabla h)_{L^2}\).

We analyze the following weak form of the wave equation with zero Dirichlet boundary conditions: 
\begin{equation} \label{eq:4_MultiDweakform}
\begin{split}
    (u_{tt},v)_{L^2} + c^2 (\nabla u,\nabla v)_{L^2} &+ D (u_{t},v)_{L^2} + G (\nabla u_{t}, \nabla v)_{L^2} = 0, \quad \forall v\in V,  \\
    u(x,0) &= u_0, \qquad
    u_t(x,0) = u_{00}.
    \end{split}
\end{equation} We use the same time discretization scheme seen in Equation \eqref{eq:4_WeakFormFEscheme} and project onto a standard FE space \(V^h \subset V\):
\begin{align}\label{eq:4_MultiDweakformFEscheme}
 (\p\p u_h^n,v_h)_{L^2} + c^2 (\nabla \widehat{u}^n_{h},\nabla v_{h})_{L^2} &+ D (\p \overline{u}^n_h,v_h)_{L^2} + G (\p \nabla \overline{u}^n_{h}, \nabla v_{h})_{L^2} = 0, \quad \forall v_h \in V^h,
 \end{align}
where \(u_h^1, u^2_h \in V^h\) are given. 
Next we look at the ROM of Equation \eqref{eq:4_MultiDweakformFEscheme} using the data set \(\{u^n\}_{n=1}^N\) to form the POD basis, \(\{\varphi_j\}_{j=1}^r \subset V^h\), with either the standard POD method or the new DDQ approach. In this work, we take the POD space to be \(L^2(\Omega)\) in all cases. Let \(V^h_r = \Span\{\varphi_j\}_{j=1}^r\). Then the POD ROM is 
\begin{align}\label{eq:4_PODROMweakform}
    (\p\p u_r^n,v_r)_{L^2} + c^2 (\nabla \widehat{u}^n_{r},\nabla v_{r})_{L^2} &+ D (\p \overline{u}^n_r,v_r)_{L^2} + G (\p \nabla \overline{u}^n_{r}, \nabla v_{r})_{L^2} = 0 \quad \forall v_r \in V^h_r, \nonumber \\
    u^1_r &= \Pi^X_r u^1_h, \\
    u^2_r &= \Pi^X_r u^2_h. \nonumber
\end{align}

\begin{lem} \label{lem:4_DDQavgIPresults}
Let \(\Delta t > 0\), \(Z\) be an inner product space, and \(\{z^n\}_{n=1}^N \subset Z\). Then for \( n = 2, \ldots, N-1\), we have 
\begin{equation}
    (\p\p z^n, \p \overline{z}^n)_Z = \p\left(\frac{1}{2}\|\pn z^n\|_Z^2 \right)\label{eq:4_DDQDQAvg}
\end{equation}
\begin{equation}
     (\widehat{z}^n, \p \overline{z}^n)_Z = \p\left(\frac{1}{2}\|\overline{z}^n\|_Z^2 \right)\label{eq:4_HatLineAvg}.
\end{equation}
\end{lem}
\begin{proof}
    We prove only Equation \eqref{eq:4_DDQDQAvg} as the proof of the Equation \eqref{eq:4_HatLineAvg} is very similar. Notice that
    \begin{align*}
        (\p\p z^n, \p \overline{z}^n)_Z &= \frac{1}{2\Delta t}\left(\frac{z^{n+1} - 2z^n + z^{n-1}}{\Delta t},\frac{z^{n+1} - z^{n-1}}{\Delta t}\right)_Z\\ &= \frac{1}{2\Delta t}\left(\pn z^{n+1} - \pn z^n, \pn z^{n+1} + \pn z^n \right)_Z\\
        &= \frac{1}{2\Delta t}\left(\|\pn z^{n+1}\|^2_Z - \|\pn z^n\|^2_Z\right)\\
        &= \p\left(\frac{1}{2}\|\pn z^n\|_Z^2\right).
    \end{align*}
\end{proof}
It is useful to define an energy quantity for this system. We do so below, and we present an equality governing the discrete time rate of change for the energy. 
\begin{prop} \label{prop:4_EnergyDef}
    For the discrete FE equation \eqref{eq:4_MultiDweakformFEscheme} and POD ROM equation \eqref{eq:4_PODROMweakform}, if the energy is defined by 
    \begin{equation} \label{eq:4_FEenergy}
        E(u_h^n) = \frac{1}{2}\|\pn u_h^n\|_{L^2}^2 + \frac{1}{2}c^2\|\nabla \overline{u}^n_h\|^2_{L^2}
    \end{equation}
    \begin{equation} \label{eq:4_ROMenergy}
        E(u^n_r) = \frac{1}{2}\|\pn u^n_r\|_{L^2}^2 + \frac{1}{2}c^2\|\nabla \overline{u}_r^n\|^2_{L^2},
    \end{equation}
    then they satisfy \begin{equation} \label{eq:4_FEenergyDiff}
    \p E(u_h^n) = - D\|\p \overline{u}^n_h\|^2_{L^2} - G\|\p \nabla \overline{u}_h^n\|^2_{L^2}
    \end{equation}
    \begin{equation} \label{eq:4_ROMenergyDiff}
    \p E(u^n_r) = - D\|\p \overline{u}^n_r\|^2_{L^2} - G\|\p \nabla \overline{u}^n_r\|^2_{L^2},
\end{equation}
respectively.
\end{prop} 
The proof follows directly from Lemma \ref{lem:4_DDQavgIPresults} and letting \(v_h = \p \overline{u}_h^n\) and \(v_r = \p \overline{u}_r^n\). One can easily see that if both damping coefficients are zero then the energy is constant, which we expect from an undamped wave equation. 

\subsection{Preliminary Error Analysis}\label{sec:PrelErrorAnalysis}

To analyze the error, we split it in the normal way, as
\begin{equation}\label{eq:4_ErrorDef}
    e^n = u_h^n - u_r^n = (u_h^n - R_r u_h^n) - (u^n_r - R_r u_h^n ) = \eta^n - \phi^n_r
\end{equation}
where \(\eta^n\) is the POD projection error, \(\phi^n_r\) is the discretization error, and \(R_r : V^h \to V^h_r\) is the Ritz projection defined by 
\begin{equation} \label{eq:4_RitzDef}
        (\nabla (w - R_r w), \nabla v_r)_{L^2} = 0
    \end{equation}
    for all \(v_r \in V^h_r\) and any \(w \in V^h\). 
Subtracting Equation \eqref{eq:4_PODROMweakform} from Equation \eqref{eq:4_MultiDweakformFEscheme} and applying Equation \eqref{eq:4_ErrorDef} yields
\begin{align} \label{eq:4_ErrorWeakForm}
    &(\p\p \phi_r^n,v_r)_{L^2} + c^2 (\nabla \widehat{\phi}^n_{r},\nabla v_{r})_{L^2} + D (\p \overline{\phi}^n_r,v_r)_{L^2} + G (\p \nabla \overline{\phi}^n_{r}, \nabla v_{r})_{L^2} \nonumber \\
    & \quad = (\p\p \eta^n,v_r)_{L^2} + c^2 (\nabla \widehat{\eta}^n,\nabla v_{r})_{L^2} + D (\p \overline{\eta}^n,v_r)_{L^2} + G (\p \nabla \overline{\eta}^n, \nabla v_{r})_{L^2} \quad \forall v_r \in V_r^h.
\end{align}
Let \(C_p > 0\) be the constant so that Poincar\'e inequality \(C_p \|\phi \|^2_{L^2} \leq \|\phi \|^2_{H^1_0}\) holds for all \(\phi \in H^1_0(\Omega)\). Lemma \ref{lem:4_phiEnergyBound} proves a bound for the discretization error in terms of the POD data error. It is important for proving the pointwise and energy error bounds in Section \ref{sec:4_ROMErrorBounds}

\begin{lem} \label{lem:4_phiEnergyBound}
Let \(\phi_r^n\) be the discretization error and \(\eta^n\) be the POD data error as defined in Equation \eqref{eq:4_ErrorDef} and let Equation \eqref{eq:4_ErrorWeakForm} define the relationship between \(\phi^n_r\) and \(\eta^n\). Then 
\begin{align}
    \max_{2\leq j \leq N} E(\phi_r^j) &\leq E(\phi_r^2) + \frac{1}{D + 2C_pG} \sum_{n=2}^{N-1} \Delta t\|\p\p \eta^n\|^2_{L^2} + D\sum_{n=2}^{N-1} \Delta t \|\p \overline{\eta}^n\|^2_{L^2} \label{eq:4_phiEnergyBound}
\end{align}
\begin{proof}
    To prove this, first notice that Equation \eqref{eq:4_ErrorWeakForm} with \(v_r = \p \overline{\phi}^n_r\) can be rewritten as 
    \begin{align*}
        &\frac{1}{2}\p E(\phi_r^n) + D\|\p \overline{\phi}^n_r\|^2_{L^2} + G\|\p \nabla \overline{\phi}^n_r\|^2_{L^2} \\
        & \quad \qquad = (\p\p \eta^n,\p \overline{\phi}^n_r)_{L^2} + c^2 (\nabla \widehat{\eta}^n,\p \nabla \overline{\phi}^n_r)_{L^2} + D (\p \overline{\eta}^n,\p \overline{\phi}^n_r)_{L^2} + G (\p \nabla \overline{\eta}^n, \p\nabla  \overline{\phi}^n_r)_{L^2}.
    \end{align*}
   Then the Ritz projection eliminates the \(\eta^n\) gradient terms from the RHS yielding
    \begin{align*}
        \p E(\phi_r^n) &\leq 2(\p\p \eta^n,\p \overline{\phi}^n_r)_{L^2} + 2D (\p \overline{\eta}^n,\p \overline{\phi}^n_r)_{L^2} - 2D\|\p \overline{\phi}^n_r\|^2_{L^2} - 2G\|\p \nabla \overline{\phi}^n_r\|^2_{L^2}.
    \end{align*}
    We then use Cauchy-Schwartz and Young's inequality twice with constants \(\delta_1\) and \(\delta_2\) to obtain
    \begin{align*}
        \p E(\phi_r^n) &\leq \frac{1}{\delta_1}\|\p\p \eta^n\|^2_{L^2} + \delta_1 \|\p \overline{\phi}^n_r\|^2_{L^2} + \frac{D}{\delta_2} \|\p \overline{\eta}^n\|^2_{L^2} + \delta_2 D\|\p \overline{\phi}^n_r\|^2_{L^2}\\ &\quad \qquad - 2D\|\p \overline{\phi}^n_r\|^2_{L^2} - 2G\|\p \nabla \overline{\phi}^n_r\|^2_{L^2}.
    \end{align*}
    Using the fact that \(C_p\|\phi^n_r\|^2_{L^2} \leq \|\nabla \phi^n_r\|^2_{L^2}\), we have \begin{align*}
        \p E(\phi_r^n) \leq \frac{1}{\delta_1}\|\p\p \eta^n\|^2_{L^2} + \frac{D}{\delta_2} \|\p \overline{\eta}^n\|^2_{L^2} + (\delta_1 + \delta_2 D - 2D - 2C_pG)\|\p \overline{\phi}^n_r\|^2_{L^2}.
    \end{align*}
    Setting \(\delta_1 = D + 2C_pG\) and \(\delta_2 = 1\) yields
    \begin{align*}
        \p E(\phi_r^n) \leq \frac{1}{D + 2C_pG}\|\p\p \eta^n\|^2_{L^2} + D \|\p \overline{\eta}^n\|^2_{L^2}.
    \end{align*}
    Finally summing from \(n = 2\) to \(n = j-1\) yields 
    \begin{align*} 
        E(\phi_r^j) &\leq  E(\phi_r^2) + \frac{1}{D + 2C_pG} \sum_{n=2}^{j-1}\Delta t\|\p\p \eta^n\|^2_{L^2} + D \sum_{n=2}^{j-1}\Delta t\|\p \overline{\eta}^n\|^2_{L^2}.
    \end{align*} 
    Take the maximum over all \(j\) to prove the result. 
\end{proof}
\end{lem}
\begin{remark} It is important to note that in Lemma \ref{lem:4_phiEnergyBound} it is possible for one of \(D\) and \(G\) to be zero but not both. The structure of Theorems \ref{thm:4_eEnergyBound} and \ref{thm:4_ePWBound} does not change if one is zero, only the constant \(C\) changes in both. 
\end{remark}

\subsection{ROM Pointwise and Energy Error Bounds} \label{sec:4_ROMErrorBounds}
In this section, we prove new pointwise and energy error bounds for the POD-ROM. In the following theorems, the value of \(C\) does not depend on any discretization parameters. It does, however, depend on the size of the damping parameters. We will explore the value \(C\) computationally in later sections. 

\begin{theorem} \label{thm:4_eEnergyBound}
    Using the \(L^2(\Omega)\) POD basis, the maximum energy of the error in the POD-ROM is bounded by
    \begin{equation} \label{eq:4_eEnergyBound}
        \max_{2\leq j \leq N} E(e^j) \leq C\Biggl(E(\phi_r^2) + \sum_{k=r+1}^s \lambda^{\mathrm{DDQ}}_k\left(\|\varphi_k - R_r \varphi_k\|^2_{L^2}+ \|\varphi_k - R_r \varphi_k\|^2_{H^1_0}\right)\Biggr).
    \end{equation}
\end{theorem}
\begin{proof}
    Using the energy definition in Proposition \ref{prop:4_EnergyDef},
    \begin{align*}
        E(e^n) &= \frac{1}{2}\left(\|\pn e^n\|^2_{L^2} + c^2\|\overline{e}^n\|^2_{H^1_0}\right) \\ &\leq \|\pn \eta^n\|^2_{L^2} + c^2\|\overline{\eta}^n\|^2_{H^1_0} + \|\pn \phi_r^n\|^2_{L^2} + c^2\|\overline{\phi}_r^n\|^2_{H^1_0}. 
    \end{align*}
    By Lemma \ref{lem:4_phiEnergyBound},
    \begin{align*}
        E(e^n) &\leq E(\phi_r^2) + \|\pn \eta^n\|^2_{L^2} + c^2\|\overline{\eta}^n\|^2_{H^1_0} + \frac{1}{D + 2C_pG} \sum_{i=2}^{N-1} \Delta t\|\p\p \eta^i\|^2_{L^2} + D\sum_{i=2}^{N-1} \Delta t \|\p \overline{\eta}^i\|^2_{L^2} \\ &\leq E(\phi_r^2) + \|\pn \eta^n\|^2_{L^2} + c^2\|\overline{\eta}^n\|^2_{H^1_0} + \frac{1}{D + 2C_pG} \sum_{i=2}^{N-1} \Delta t\|\p\p \eta^i\|^2_{L^2} + TD \max_i \|\p \overline{\eta}^i\|^2_{L^2}.
    \end{align*}
    Next, apply the results of Lemma \ref{lem:3_zDDQErrorBounds} to yield
    \begin{align*}
        \max_n E(e^n) &\leq E(\phi^2_r) + C_3 \left(\|\p \eta^1\|^2_{L^2}+ \sum_{i=2}^{N-1} \Delta t \|\p\p \eta^i\|^2_{L^2} \right) 
        \\ &  \qquad + c^2 C_2\left(\| \eta^1\|^2_{H^1_0} + \|\p \eta^1\|^2_{H^1_0} + \sum_{i=2}^{N-1} \Delta t \|\p\p \eta^i\|^2_{H^1_0}\right) 
        \\ & \qquad  +\frac{1}{D + 2C_pG} \sum_{i=2}^{N-1} \Delta t\|\p\p \eta^i\|^2_{L^2} + TDC_3\left( \|\p \eta^1\|^2_{L^2} + \sum_{i=2}^{N-1} \Delta t \|\p\p \eta^i\|^2_{L^2}\right).
    \end{align*}
    Adding in positive terms and combining like terms yields
    \begin{align*}
        \max_n E(e^n) &\leq \left(C_3 + DTC_3 + \frac{1}{D + 2C_pG}\right) \left(\| \eta^1\|^2_{L^2}+\|\p \eta^1\|^2_{L^2}+ \sum_{i=2}^{N-1} \Delta t \|\p\p \eta^i\|^2_{L^2} \right) 
        \\ & \qquad + c^2 C_2\left(\| \eta^1\|^2_{H^1_0} + \|\p \eta^1\|^2_{H^1_0} + \sum_{i=2}^{N-1} \Delta t \|\p\p \eta^i\|^2_{H^1_0}\right) + E(\phi_r^2).
    \end{align*}
    Finally applying Lemma \ref{lem:3_DDQPODErrorFormulas}, we have that for some constant \(C\) 
    \begin{align*}
        \max_n E(e^n) &\leq C\Biggl(E(\phi_r^2) + \sum_{k=r+1}^s \lambda^{\mathrm{DDQ}}_k\left(\|\varphi_k - R_r \varphi_k\|^2_{L^2}+ \|\varphi_k - R_r \varphi_k\|^2_{H^1_0}\right) \Biggr).
    \end{align*}
    Thus Equation \eqref{eq:4_eEnergyBound} is proven. 
\end{proof}

\begin{theorem} \label{thm:4_ePWBound}
    Using the \(L^2(\Omega)\) POD basis, the maximum pointwise error for the POD-ROM is bounded by 
    \begin{align} \label{eq:4_ePWBound}
        \max_{1 \leq j \leq N} \|e^j\|^2_{L^2} &\leq C\left(\|\phi^1_r\|^2_{L^2} + E(\phi_r^2) + \sum_{k=r+1}^s \lambda^{\mathrm{DDQ}}_k \|\varphi_k - R_r \varphi_k\|^2_{L^2}\right)
    \end{align}
\end{theorem}
\begin{proof}
    By Lemma \ref{lem:2_DQPointwiseErrorBounds},
    \begin{align*}
        \max_n \|\phi^n\|^2_{L^2} &\leq C_1\left(\|\phi^1\|^2_{L^2} + \sum_{k = 2}^{N} \Delta t\|\p^-\phi^k_r\|^2_{L^2}\right)\\
        & \leq C_1\left(\|\phi^1\|^2_{L^2} + T\max_k \|\p^- \phi^k_r\|^2_{L^2}\right).
    \end{align*}
    And by Lemma \ref{lem:4_phiEnergyBound}, 
    \begin{align*}
        \max_n \|\phi^n\|^2_{L^2} &\leq C_1\Biggl(\|\phi^1_r\|^2_{L^2} + T\left(\|\pn \phi^2_r\|_{L^2}^2 + c^2 \| \overline{\phi}^2_{r}\|_{H^1_0}^2\right) \\& \quad \qquad +\frac{T}{D + 2C_pG} \sum_{i=2}^{N-1} \Delta t\|\p\p \eta^i\|^2_{L^2} + TD\sum_{i=2}^{N-1} \Delta t \|\p \overline{\eta}^i\|^2_{L^2}\Biggr).
    \end{align*}
    Using Lemma \ref{lem:3_zDDQErrorBounds}, we have
    \begin{align*}
        \max_n \|\phi^n\|^2_{L^2} &\leq C_1\Biggl(\|\phi^1_r\|^2_{L^2} + T\left(\|\pn \phi^2_r\|_{L^2}^2 + c^2 \| \overline{\phi}^2_{r}\|_{H^1_0}^2\right) \\& \quad \qquad +\frac{T}{D + 2C_pG} \sum_{i=2}^{N-1} \Delta t\|\p\p \eta^i\|^2_{L^2} + T^2DC_3\left( \|\p \eta^1\|^2_{L^2} + \sum_{i=2}^{N-1} \Delta t \|\p\p \eta^i\|^2_{L^2}\right) \Biggr).
    \end{align*}
    Since
    \begin{align*}
        \|e^n\|^2_{L^2} \leq 2\|\eta^n\|^2_{L^2} + 2\|\phi^n_r\|^2_{L^2},
    \end{align*}
    we have that 
    \begin{align*}
        \max_n \|e^n\|^2_{L^2} &\leq 2\left(C_2 + T^2DC_1C_3 + \frac{TC_1}{D+2C_pG}\right)\left(\| \eta^1\|^2_{L^2}+\|\p \eta^1\|^2_{L^2}+ \sum_{i=2}^{N-1} \Delta t \|\p\p \eta^i\|^2_{L^2} \right) \\
        & \quad \qquad +2 C_1\|\phi^1_r\|^2_{L^2} + 2TC_1E(\phi^2_r).
    \end{align*}
    Applying Lemma \ref{lem:3_DDQPODErrorFormulas}, we have that for some constant \(C\) 
    \begin{align*}
        \max_{1 \leq j \leq N} \|e^j\|^2_{L^2} &\leq C\left(\|\phi^1_r\|^2_{L^2} + E(\phi^2_r) + \sum_{k=r+1}^s \lambda^{\mathrm{DDQ}}_k \|\varphi_k - R_r \varphi_k\|^2_{L^2}\right).
    \end{align*}
\end{proof}

\section{Computational Results} \label{sec:5_all}

In this section, we present numerous computational results. Section \ref{sec:5_PODdataComps} covers results exploring the singular values for the Standard POD method and the DDQ approach. We also verify the data error formulas for both methods. In Section \ref{sec:5_DDQPODROM Bounds}, we explore the bounds from Theorems \ref{thm:4_eEnergyBound} and \ref{thm:4_ePWBound}, and compare the performance of the ROM when using Standard POD and DDQ POD in Section \ref{sec:5_StVSDDQROM}. Finally, in Section \ref{sec:5_ReducedInterval}, we perform exploratory computations for the accuracy of the ROM when including only part of the interval to collect the data.

We also present the differing behaviors of the two types of damping we considered in the error analysis. In all computations only one damping constant is nonzero at a time. In Sections \ref{sec:5_PODdataComps} and \ref{sec:5_ReducedInterval}, we choose one value of each damping parameter to show results comparing the two. For the viscous damping, we choose \(D = 0.1\) as the test value and for the Kelvin-Voigt damping, we choose \(G = 0.001\). At these values each damping has a visible effect on the time evolution of the wave. The way they interact with both methods of POD leads to different singular value decays and how many POD basis functions are required for accurate approximation. In Section \ref{sec:5_DDQPODROM Bounds}, we present results for the scaling factor in Theorems \ref{thm:4_eEnergyBound} and \ref{thm:4_ePWBound}  and in Section \ref{sec:5_StVSDDQROM}we explore the magnitude of the energy and pointwise errors for the two methods for a range of damping values.

\subsection{POD Data Computations} \label{sec:5_PODdataComps}

Here, we present computational results verifying the POD data error formulas for Standard POD and the DDQ POD method. In all examples and computations provided, we use \(X = L^2(0,1)\). We let \(T = 10\), \(\Delta t = \frac{1}{800}\), and choose 400 finite element nodes. We found that changing the total number of finite element nodes did not have a large effect on the performance of POD. 

To get the data \(\{u^j\}\), we compute the FE solution with the chosen initial condition. For the standard POD computations we choose \(\gamma_j = \Delta t\) for all \(j = 1, \ldots, N\). To compute the SVD of the POD operator, we use the method described in Section 2.2 of \cite{Fareed}. We make small modifications to the scaling of the data due to the POD weights. 

In Figure \ref{fig:5.1}, we can see that the singular value decay when \(D = 0.1\) and \(G=0\) is very slow for both methods, but slightly slower for the DDQ POD. The magnitude of the singular values is also larger for that method.

\begin{figure}[htb]
\centering
\begin{subfigure}{.5\textwidth}
  \centering
  \includegraphics[width=1\linewidth]{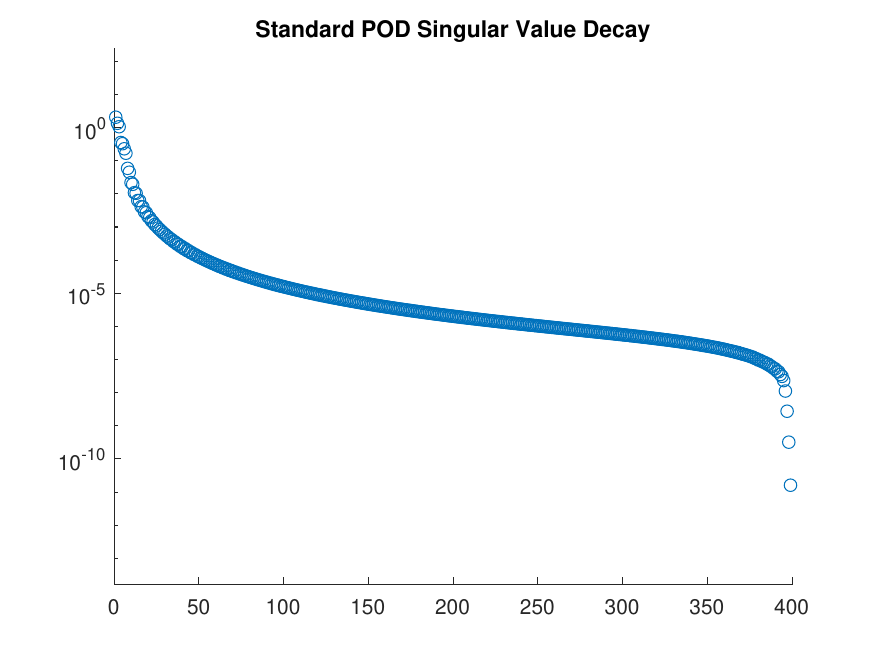}
  \caption{Standard POD}
  \label{fig:sub1}
\end{subfigure}%
\begin{subfigure}{.5\textwidth}
  \centering
  \includegraphics[width=1\linewidth]{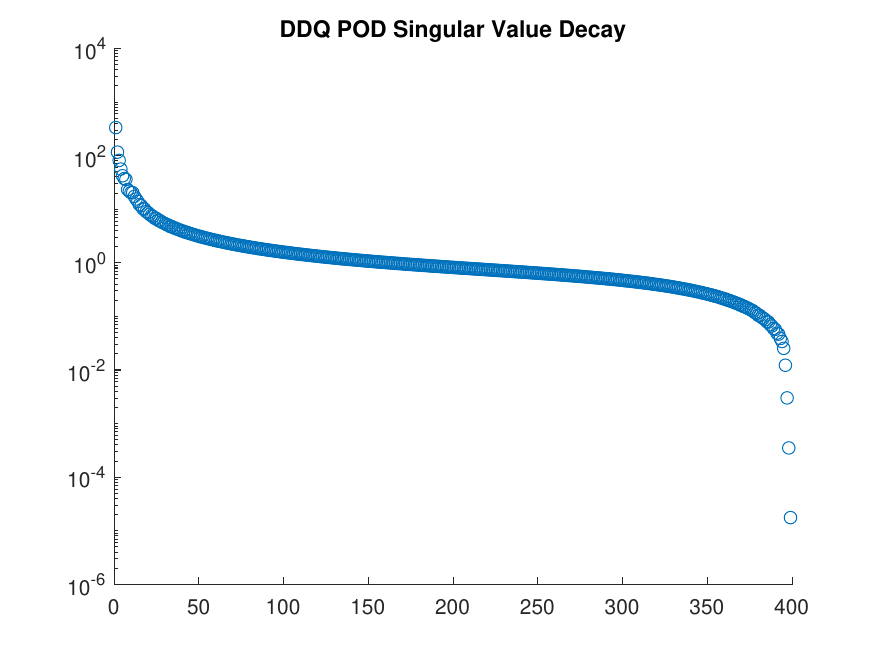}
  \caption{DDQ POD}
  \label{fig:sub2}
\end{subfigure}
\caption{POD singular values for \(D = 0.1\) and \(G = 0\)}
\label{fig:5.1}
\end{figure}

Figure \ref{fig:5.2} shows the contrasting behavior of the singular values when \(G = 0.001\) and \(D = 0\). Note that in \ref{fig:5.2}, we only plot the first 75 singular values. This is due to them leveling off at numerical round off errors at around \(10^{-10}\). The Kelvin-Voigt damping term has a much stronger effect on the information content than the viscous damping term. For both types of damping, we see the Standard POD method has a slightly faster decay for the singular values.

\begin{figure}[htb]
\centering
\begin{subfigure}{.5\textwidth}
  \centering
  \includegraphics[width=1\linewidth]{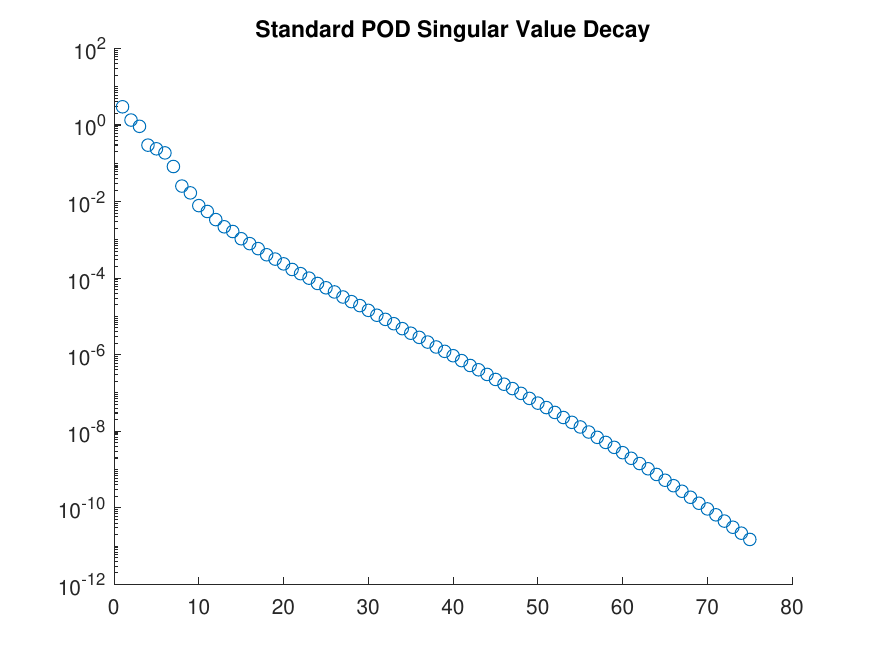}
  \caption{Standard POD}
  \label{fig:sub1}
\end{subfigure}%
\begin{subfigure}{.5\textwidth}
  \centering
  \includegraphics[width=1\linewidth]{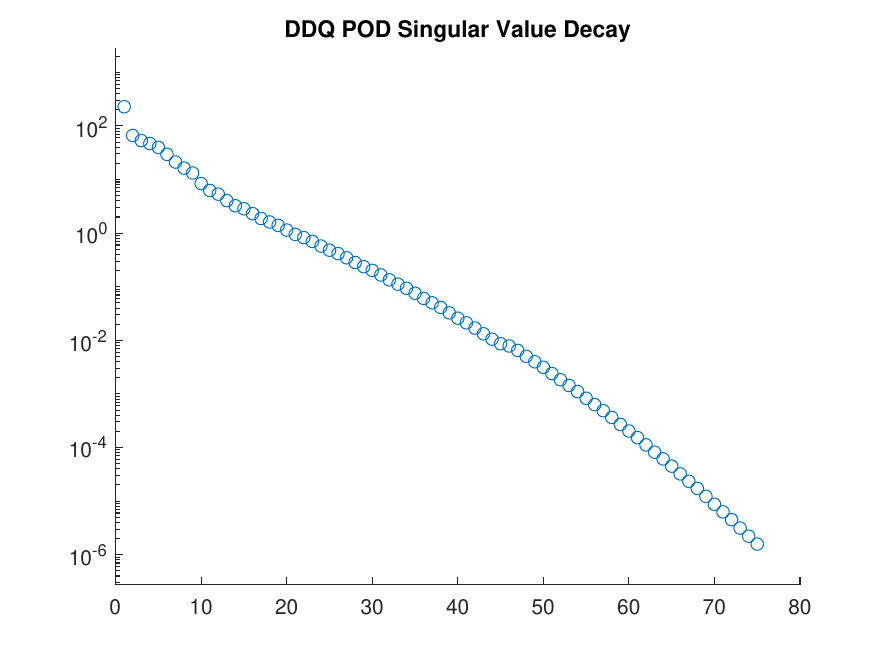}
  \caption{DDQ POD}
  \label{fig:sub2}
\end{subfigure}
\caption{POD singular values when \(G = 0.001\) and \(D = 0\)}
\label{fig:5.2}
\end{figure}

Tables \ref{tab:5.1} and \ref{tab:5.2} show the POD data error formulas from Sections \ref{sec:2_StPOD} and \ref{sec:3_DDQPODmethod} being applied when \( G = 0.001\) and \( D = 0\). The data errors are computed with respect to the given norm. The singular value errors are computed with the right hand side in Equation \eqref{eq:2_StPODErrorFormula} and Lemmas \ref{lem:2_StPODExtendedErrors} and \ref{lem:3_DDQPODErrorFormulas}. For example, the last two columns of Table \ref{tab:5.1} are computed as 
\begin{equation}
    H^1_0 \text{ Actual Error } = \sum_{j=1}^N \Delta t\|u^j - \Pi^X_r u^j\|^2_{H^1_0}, \qquad H^1_0 \text{ Error Formula } = \sum_{k = r + 1}^s \lambda_k\|\varphi_k\|^2_{H^1_0}
\end{equation}
with \(Y = H^1_0\) in Lemma \ref{lem:2_StPODExtendedErrors}. The results in Tables \ref{tab:5.1} and \ref{tab:5.2} are accurate up to many decimal places verifying the data error formulas.
\begin{table}[htb]
	\renewcommand{\arraystretch}{1.25}
	\begin{center}	
		\begin{tabular}{c|c|c|c|c}
			\hline
			\(r\) value & Equation & Error Norm & Actual Error & Error Formula \\
			\hline
			10 & \eqref{eq:2_StPODErrorFormula}  & $X=L^2(\Omega)$ & 5.18E-05 & 5.18E-05 \\
			& \eqref{eq:2_StPODErrorY} & $Y=H_0^1(\Omega)$ & 7.46E-02 & 7.46E-02 \\
			\hline
			20 & \eqref{eq:2_StPODErrorFormula} &  $X=L^2(\Omega)$ & 6.82E-08 & 6.82E-08 \\
			& \eqref{eq:2_StPODErrorY} & $Y=H_0^1(\Omega)$ & 4.72E-04 & 4.72E-04 \\
			\hline
			40 & \eqref{eq:2_StPODErrorFormula} & $X=L^2(\Omega)$ & 1.14E-12 & 1.14E-12 \\
			& \eqref{eq:2_StPODErrorY} & $Y=H_0^1(\Omega)$ & 8.57E-08 & 8.57E-08 \\
			\hline
            60 & \eqref{eq:2_StPODErrorFormula} & $X=L^2(\Omega)$ & 8.15E-18 & 8.15E-18 \\
			& \eqref{eq:2_StPODErrorY} & $Y=H_0^1(\Omega)$ & 2.10E-12 & 2.10E-12 \\
			\hline
		\end{tabular}
		\caption{ \label{tab:5.1} Actual error versus error formulas from Equation \eqref{eq:2_StPODErrorFormula} and Lemma \ref{lem:2_StPODExtendedErrors} for standard POD with $X = L^2(\Omega)$, \(D = 0\), and \(G=0.001\).}
	\end{center}
\end{table}
\begin{table}[htb]
	\renewcommand{\arraystretch}{1.25}
	\begin{center}	
		\begin{tabular}{c|c|c|c|c}
			\hline
			\(r\) value & Equation & Error Norm & Actual Error & Error Formula \\
			\hline
			10 & \eqref{eq:3_DQPODError}  & $X=L^2(\Omega)$ & 1.20E+02 & 1.20E+02 \\
			& \eqref{eq:3_DQPODErrorY} & $Y=H_0^1(\Omega)$ & 3.11E+05 & 3.11E+05 \\
			\hline
			20 & \eqref{eq:3_DQPODError} &  $X=L^2(\Omega)$ & 3.17 & 3.17 \\
			& \eqref{eq:3_DQPODErrorY} & $Y=H_0^1(\Omega)$ & 5.21E+04 & 5.21E+04 \\
			\hline
			40 & \eqref{eq:3_DQPODError} & $X=L^2(\Omega)$ & 1.26E-03 & 1.26E-03 \\
			& \eqref{eq:3_DQPODErrorY} & $Y=H_0^1(\Omega)$ & 2.36E+02 & 2.36E+02 \\
			\hline
            60 & \eqref{eq:3_DQPODError} & $X=L^2(\Omega)$ & 5.06E-08 & 5.06E-08 \\
			& \eqref{eq:3_DQPODErrorY} & $Y=H_0^1(\Omega)$ & 1.22E-02 & 1.22E-02 \\
			\hline
		\end{tabular}
		\caption{ \label{tab:5.2} Actual error versus error formulas from Lemma \ref{lem:3_DDQPODErrorFormulas} for DDQ POD with $X = L^2(\Omega)$, \(D = 0\), and \(G=0.001\).}
	\end{center}
\end{table}

The results for \(D = 0.1\) and \(G = 0\) were similar, albeit much larger for both methods of POD, and are not presented. The difference in magnitude of the singular values between the two POD approaches is likely due to the magnitude of the norm for the 2nd difference quotients. This was also seen in \cite{Herkt}.

\subsection{ROM Computations}\label{sec:5_ROMcomps}

We split this section into three parts. The first covers the ROM error bounds from Section \ref{sec:4_ROMErrorBounds} and the second compares standard POD to DDQ POD for the ROM construction by considering the maximum energy errors and \(L^2\) pointwise errors. 

In the third, we keep the same testing interval \([0, 10]\), but we reduce the training interval where we collect the snapshots to \([0, T_t]\) where \(T_t \leq 10\). We then simulate the POD ROM over the entire test interval and compare the final time errors for each method of POD. This tests the long term accuracy of the ROMs for simulating into the future. 

\subsubsection{DDQ POD ROM Error Bounds} \label{sec:5_DDQPODROM Bounds}

First, we explore the actual size of the constants in Theorems \ref{thm:4_eEnergyBound} and \ref{thm:4_ePWBound} using the DDQ POD method. We do this at various values for each of the damping parameters and at different values of \(r\). In each of these tests, only one damping parameter is nonzero.

For Theorem \ref{thm:4_eEnergyBound}, we calculate the scaling factor with
\begin{equation}
    C = \frac{\max_n E(e^n)}{\left(E(\phi^2_r) + \sum_{k=r+1}^s \lambda^{\mathrm{DDQ}}_k\left(\|\varphi_k - R_r \varphi_k\|^2_{L^2}+ \|\varphi_k - R_r \varphi_k\|^2_{H^1_0}\right)\right)}.
\end{equation}
The results are shown in Table \ref{tab:5.3}. The wide range of damping values allows us to see a few patterns emerge for each type. As we vary the viscous damping parameter, \(D\), the scaling factor is remarkably stable whereas the Kelvin-Voigt damping, \(G\), exhibits similar behavior but seems to have two different scales. For \(G\leq 0.01\), the scaling factor is stable for each \(r\) value. However, when \(G \geq 0.01\) we start to see large changes in the scaling factor. This may be due to the number of oscillatory modes decreasing to less than 32 when \(G = 0.01\) decreasing the magnitude of the bound.

\begin{table}[htb]
  \begin{center}       
       \begin{tabular}{c|c|c||c|c|c}
\hline     D       & \(r = 20\) & \(r = 40\)& G     & \(r = 10\) & \(r = 20\) \\ \hline
0.00001 & 1.43-E-09  & 5.48E-10 & 0.00001 & 7.21E-08   & 4.95E-09   \\ \hline
0.0001  & 1.43-E-09  & 5.46E-10 & 0.0001  & 1.35E-07   & 3.02E-08   \\ \hline
0.001   & 1.38-E-09  & 5.32E-10 & 0.001   & 4.44E-07   & 8.25E-07   \\ \hline
0.01    & 1.04-E-09  & 4.42E-10 & 0.01    & 5.12E-06   & 1.21E-03   \\ \hline
0.1     & 7.66-E-10  & 6.53E-10 & 0.1     & 1.12E-03   & 5.30E-05   \\ \hline
       \end{tabular}
  \end{center}
  \caption{Scaling factor for Theorem \ref{thm:4_eEnergyBound} at various damping values.}
\label{tab:5.3}
\end{table}

For Theorem \ref{thm:4_ePWBound}, we calculate the scaling factor with \begin{equation}
    C = \frac{\max_n \|e^n\|^2_{L^2}}{\left(\|\phi^1_r\|^2_{L^2} + E(\phi^2_r) + \sum_{k=r+1}^s \lambda^{\mathrm{DDQ}}_k \|\varphi_k - R_r \varphi_k\|^2_{L^2}\right)}.
\end{equation}
For the pointwise \(L^2\) error, the scaling factor is once again very stable for the viscous damping. The Kelvin-Voigt damping shows better stability within two or three magnitudes compared to the variability for the energy bounds.

\begin{table}[htb]
  \begin{center}       
       \begin{tabular}{c|c|c||c|c|c}
\hline     D       & \(r = 20\) & \(r = 40\) & G       & \(r = 10\) & \(r = 20\) \\ \hline
0.00001 & 4.94E-06   & 2.75E-06   & 0.00001 & 5.53E-06   & 2.09E-03   \\ \hline
0.0001  & 4.92E-06   & 2.74E-06   & 0.0001  & 5.19E-07   & 3.20E-03   \\ \hline
0.001   & 4.73E-06   & 2.65E-06   & 0.001   & 1.09E-06   & 1.09E-01   \\ \hline
0.01    & 3.15E-06   & 1.90E-06   & 0.01    & 1.49E-05   & 3.82E-01   \\ \hline
0.1     & 1.79E-07   & 3.36E-07   & 0.1     & 1.68E-05   & 2.96E-02   \\ \hline
       \end{tabular}
  \end{center}
  \caption{Scaling factor for Theorem \ref{thm:4_ePWBound} at various damping values.}
\label{tab:5.4}
\end{table}

\subsubsection{Standard POD ROM versus DDQ POD ROM}\label{sec:5_StVSDDQROM}

Next, we compare the errors of the Standard POD ROM to the DDQ POD ROM and show graphs of the solution throughout the time interval. We once again set one damping parameter to be nonzero at a time. 
For the viscous damping, we analyze the errors at \(r = 20\) in Table \ref{tab:5.5}. This value gave good results for both POD methods.

The viscous damping showed similar behavior to the scaling factors. Both errors are remarkably static as \(D\) increases for standard POD. The DDQ POD had more interesting behavior as it started off stable when \(D\) was very small and got much more accurate as \(D\) increased.

\begin{table}[htb]
\begin{center}
\begin{tabular}{ccc||cc}
                             & \multicolumn{2}{c||}{\(L^2\) Error}           & \multicolumn{2}{c}{Energy Error}             \\ \hline
\multicolumn{1}{c|}{D}       & \multicolumn{1}{c|}{Standard POD} & DDQ POD  & \multicolumn{1}{c|}{Standard POD} & DDQ POD  \\ \hline
\multicolumn{1}{c|}{0.00001} & \multicolumn{1}{c|}{7.19E-06}     & 1.53E-02 & \multicolumn{1}{c|}{2.87E-03}     & 2.83E-01 \\ \hline
\multicolumn{1}{c|}{0.0001}  & \multicolumn{1}{c|}{7.19E-06}     & 1.52E-02 & \multicolumn{1}{c|}{2.87E-03}     & 2.82E-01 \\ \hline
\multicolumn{1}{c|}{0.001}   & \multicolumn{1}{c|}{7.18E-06}     & 1.44E-02 & \multicolumn{1}{c|}{2.87E-03}     & 2.72E-01 \\ \hline
\multicolumn{1}{c|}{0.01}    & \multicolumn{1}{c|}{7.04E-06}     & 8.68E-03 & \multicolumn{1}{c|}{2.86E-03}     & 1.89E-01 \\ \hline
\multicolumn{1}{c|}{0.1}     & \multicolumn{1}{c|}{6.73E-06}     & 2.23E-04 & \multicolumn{1}{c|}{2.78E-03}     & 7.36E-02 \\ \hline
\end{tabular}
\end{center}
\caption{Maximum \(L^2\) and energy errors for the standard POD ROM versus the DDQ POD ROM}
\label{tab:5.5}
\end{table}

For the Kelvin-Voigt damping parameter, we were able to use \(r = 10\) and get good results in Table \ref{tab:5.6}. The behavior of both methods is much less consistent here. Standard POD gets much better for both errors as \(G\) gets larger. On the other hand, the DDQ method increases in accuracy significantly slower than standard POD. 
\begin{table}[htb]
\begin{center}
\begin{tabular}{ccc||cc}
                             & \multicolumn{2}{c||}{\(L^2\) Error}           & \multicolumn{2}{c}{Energy Error}             \\ \hline
\multicolumn{1}{c|}{G}       & \multicolumn{1}{c|}{Standard POD} & DDQ POD  & \multicolumn{1}{c|}{Standard POD} & DDQ POD  \\ \hline
\multicolumn{1}{c|}{0.00001} & \multicolumn{1}{c|}{3.43E-04}     & 2.47E-02 & \multicolumn{1}{c|}{2.83E-01}     & 1.29 \\ \hline
\multicolumn{1}{c|}{0.0001}  & \multicolumn{1}{c|}{3.40E-04}     & 7.02E-04 & \multicolumn{1}{c|}{2.78E-01}     & 4.62E-01 \\ \hline
\multicolumn{1}{c|}{0.001}   & \multicolumn{1}{c|}{1.31E-04}     & 1.33E-04 & \multicolumn{1}{c|}{1.37E-01}     & 1.37E-01 \\ \hline
\multicolumn{1}{c|}{0.01}    & \multicolumn{1}{c|}{6.99E-07}     & 1.17E-04 & \multicolumn{1}{c|}{2.31E-03}     & 2.10E-01 \\ \hline
\multicolumn{1}{c|}{0.1}     & \multicolumn{1}{c|}{4.07E-11}     & 2.27E-05 & \multicolumn{1}{c|}{1.35E-05}     & 6.16E-01 \\ \hline
\end{tabular}
\end{center}
\caption{Maximum \(L^2\) and energy errors for the standard POD ROM versus the DDQ POD ROM}
\label{tab:5.6}
\end{table}

Interestingly, the two POD methods seem to swap behavior between the damping types. The standard method is very consistent as \(D\) increases while it gets much more accurate as \(G\) increases. Comparatively, the DDQ approach improves when \(D\)
increases and stays much more stable when \(G\) increases

It is clear however that for both damping types, the Standard POD ROM is equivalent or better than the DDQ POD ROM in almost all cases. This pattern continued when more basis functions or less basis functions were included.

The following graphs give a visual interpretation for some of these errors. The solid, dashed, and dotted lines represent the FE solution at times \(t = 0, 5, 10\), respectively. For the POD ROM solution we use \(*, +, \text{ and } \times\) for \(t = 0, 5, 10\), respectively. We use \(r = 10\) and \(r = 20\) when \(D = 0.1\) and \(r = 5\) and \(r = 10\) when \(G = 0.001\). Each set of of \(r\) values yields a good comparison between how effective the two methods of POD are. For each type of damping, the smaller \(r\) value yields a very inaccurate DDQ POD ROM whereas the Standard POD ROM is significantly better at that \(r\) value. This can also be seen in the error comparisons done in Tables \ref{tab:5.5} and \ref{tab:5.6}. The standard POD ROM error is always better than the DDQ POD ROM error for the viscous damping. The Kelvin-Voigt damping has similar performance for the two methods when \(r = 10\). Visually, the two methods show no difference when we go to the larger \(r\) value for each damping type. This visual confirmation of the performance of POD is interesting. We are able to represent the solution to the problem where we use \(400\) FE nodes with only \(20\) POD modes for the viscous damping and \(10\) POD modes for the Kelvin-Voigt damping demonstrating the efficiency of POD. 

The behavior of the solutions over time also provides information on why the Kelvin-Voigt damping parameter yields better results at smaller \(r\) values. The Kelvin-Voigt damping causes high frequency oscillations to die out significantly faster making it easier to represent the data over time. On the other hand, the viscous damping only causes the amplitude of the oscillation to decrease over time.

\begin{figure}[htb]
\centering
\begin{subfigure}{.5\textwidth}
  \centering
  \includegraphics[width=1\linewidth]{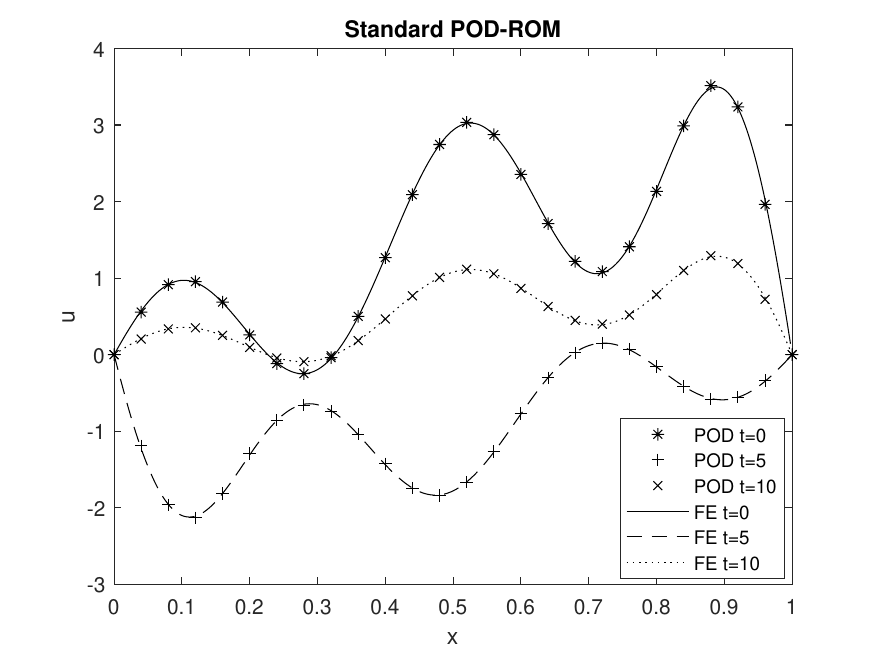}
  \caption{Standard POD ROM versus FE solution}
  \label{fig:sub1}
\end{subfigure}%
\begin{subfigure}{.5\textwidth}
  \centering
  \includegraphics[width=1\linewidth]{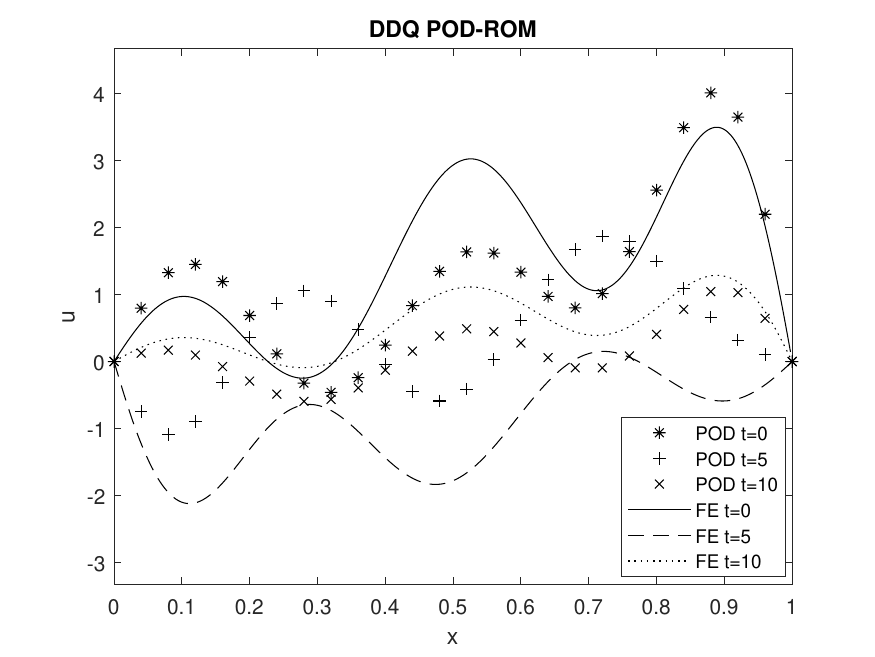}
  \caption{DDQ POD ROM versus FE solution}
  \label{fig:sub2}
\end{subfigure}
\caption{POD ROM Plots when \(D = 0.1\), \(G = 0\), and \(r = 10\)}
\label{fig:5.3}
\end{figure}
\begin{figure}[htb]
\centering
\begin{subfigure}{.5\textwidth}
  \centering
  \includegraphics[width=1\linewidth]{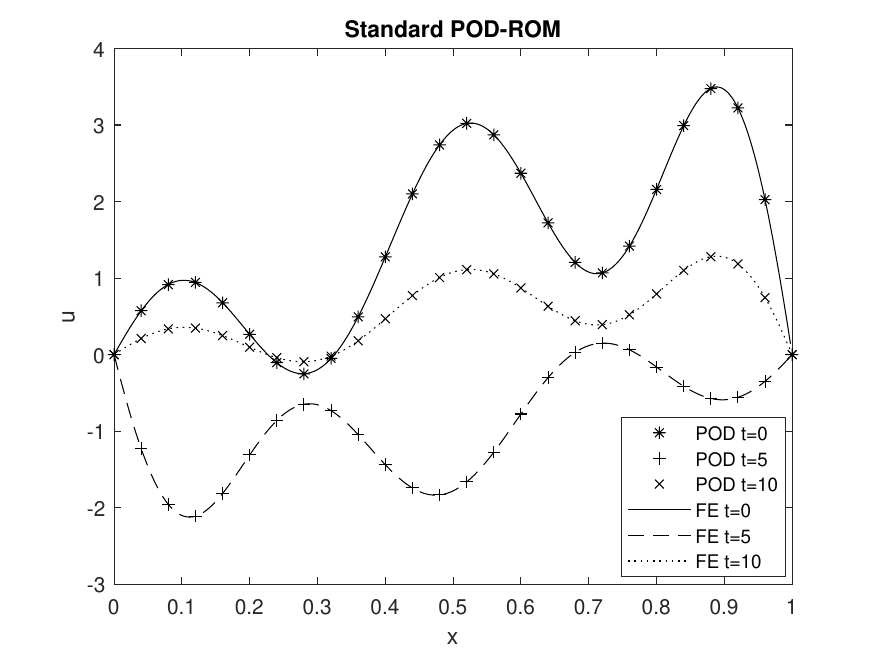}
  \caption{Standard POD ROM versus FE solution}
  \label{fig:sub1}
\end{subfigure}%
\begin{subfigure}{.5\textwidth}
  \centering
  \includegraphics[width=1\linewidth]{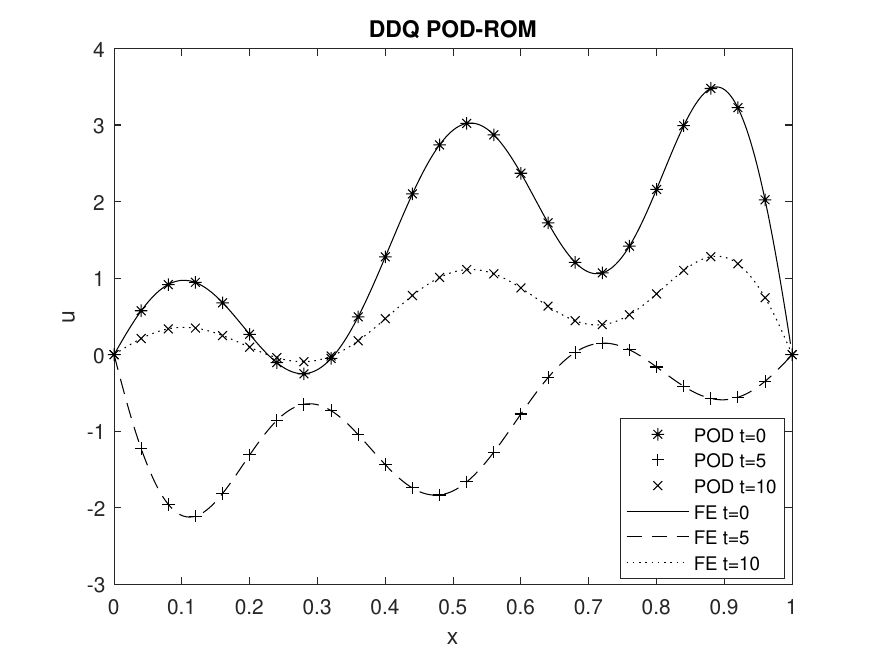}
  \caption{DDQ POD ROM versus FE solution}
  \label{fig:sub2}
\end{subfigure}
\caption{POD ROM Plots when \(D = 0.1\), \(G = 0\), and \(r = 20\)}
\label{fig:5.4}
\end{figure}

In Figure \ref{fig:5.3}, we can see that the Standard POD ROM has a few spots of inaccuracy: specifically the peak on the right for \(t = 0\) and the trough on the left for \(t = 5\). This small error in the ROM is eliminated visually in Figure \ref{fig:5.4} when \(r\) is set to 20.

\begin{figure}[htb]
\centering
\begin{subfigure}{.5\textwidth}
  \centering
  \includegraphics[width=1\linewidth]{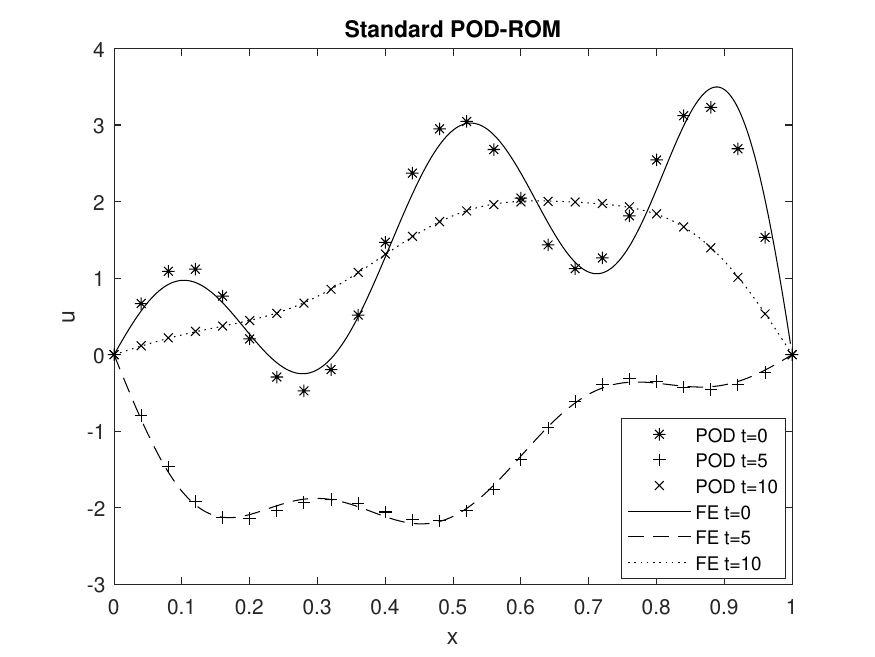}
  \caption{Standard POD ROM versus FE solution}
  \label{fig:sub1}
\end{subfigure}%
\begin{subfigure}{.5\textwidth}
  \centering
  \includegraphics[width=1\linewidth]{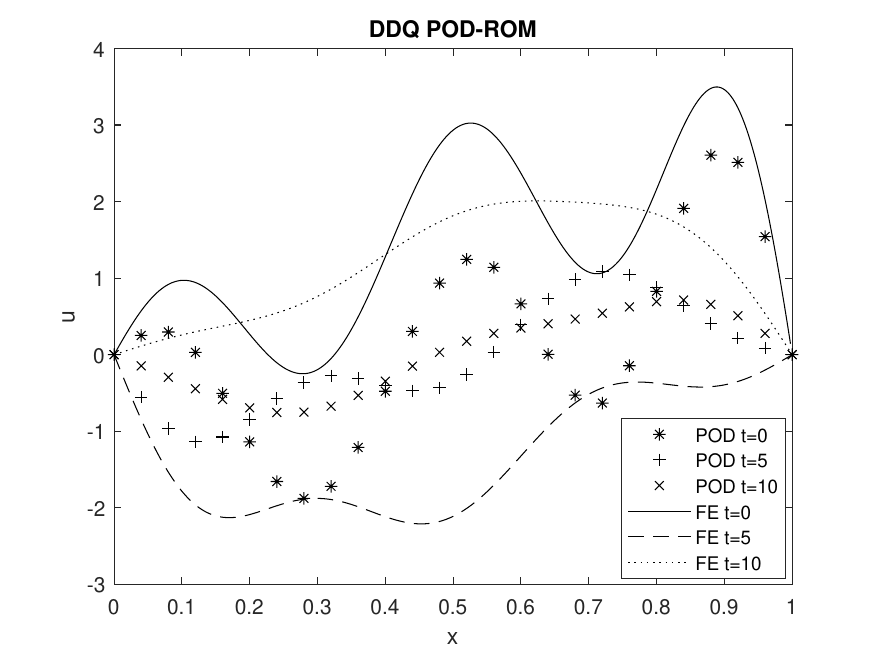}
  \caption{DDQ POD ROM versus FE solution}
  \label{fig:sub2}
\end{subfigure}
\caption{POD ROM Plots when \(D = 0\), \(G = 0.001\), and \(r = 5\)}
\label{fig:5.5}
\end{figure}
\begin{figure}[htb]
\centering
\begin{subfigure}{.5\textwidth}
  \centering
  \includegraphics[width=1\linewidth]{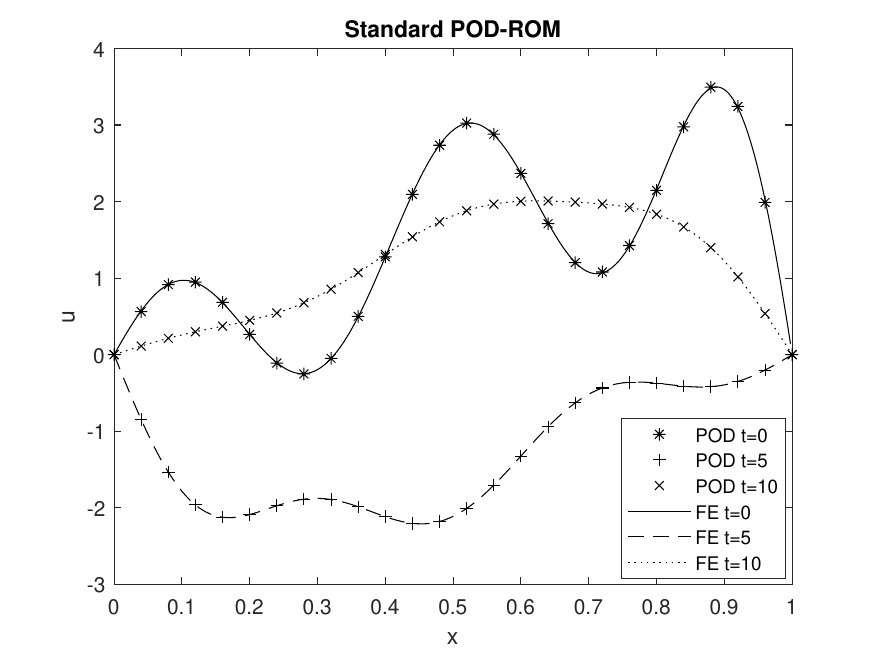}
  \caption{Standard POD ROM versus FE solution}
  \label{fig:sub1}
\end{subfigure}%
\begin{subfigure}{.5\textwidth}
  \centering
  \includegraphics[width=1\linewidth]{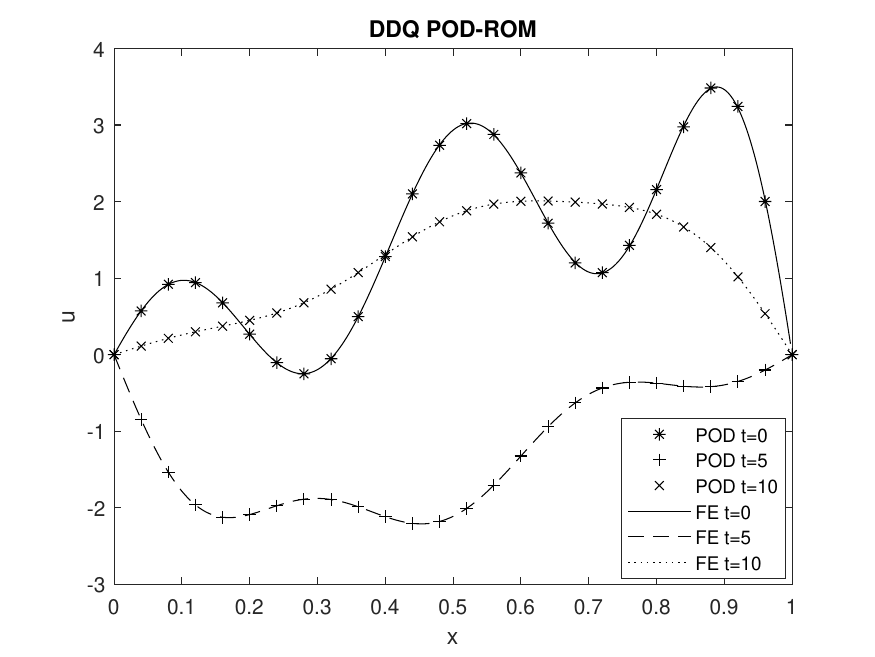}
  \caption{DDQ POD ROM versus FE solution}
  \label{fig:sub2}
\end{subfigure}
\caption{POD ROM Plots when \(D = 0\), \(G = 0.001\), and \(r = 10\)}
\label{fig:5.6}
\end{figure}

We see that for \(r = 5\), in Figure \ref{fig:5.5}, the standard POD method is able to roughly approximate the FE solution while the DDQ method would be unusable as an approximation. Since the Kelvin-Voigt damping causes high frequency oscillations to decay much quicker than low frequency ones, we see in Figure \ref{fig:5.5}(a) that the final time solution is much more accurate than the beginning time solution. This demonstrates the difficulty POD has with many frequencies of oscillation. By \(t = 10\), many of the highest frequencies have died out and POD is able to effectively represent the solution in time. Whereas, at \(t = 0\) when all of the high frequency oscillation is still present, it struggles. It also appears that the DDQ POD ROM benefits from increasing \(r\) more at small values than the standard POD ROM does. We see that for \(G\), both errors are almost the same for both methods when \(r=10\) in Figure \ref{fig:5.6}. This is clearly not true when \(r = 5\). In other exploratory computations, this same pattern was seen most often for the Kelvin-Voigt damping.

\subsubsection{Reduced Training Interval Exploration}\label{sec:5_ReducedInterval}

Next, we reduce the amount of training data POD receives when simulating over the same test interval. We do this by choosing the first \(m\) snapshots up to time \(t = T_t\) where \(T_t\) is the training time and \(m\) varies depending on the length of the interval. This is of interest as one of the major applications of a ROM is to simulate into the future based on a short period of high accuracy simulation. It is important to note that there are no theoretical foundations for these explorations.

We chose four different size training intervals: \([0,10], \ [0,5], \ [0,1]\), and \([0,.5]\). This means we include \(8001,\ 4001,\ 801\), and \(401\) snapshots respectively for each simulation. We keep the number of FE nodes at \(400\) and \(\Delta t = \frac{1}{800}\) for these tests. 
The results for each damping parameter are presented in Tables \ref{tab:5.9} and \ref{tab:5.10}.

\begin{table}[htb]
  \begin{center}       
       \begin{tabular}{c|c|c}
 \hline     Training Interval   & Standard POD \(L^2\) Error & DDQ POD \(L^2\) Error \\ \hline
\([0,10]\) & 6.04E-07  & 8.59E-07 \\ \hline
\([0,5]\)  & 6.03E-07  & 5.09E-06 \\ \hline
\([0,1]\)  & 6.40E-07  & 1.07 \\ \hline \([0,.5]\) & 3.20E-01 & 2.47E-01    \\ \hline
       \end{tabular}
  \end{center}
  \caption{Final Time \(L^2\) Error for Different Training Intervals when \(D = 0.1\) and \(r = 20\).}
\label{tab:5.9}
\end{table} 
\begin{table}[htb]
  \begin{center}       
       \begin{tabular}{c|c|c}
 \hline     Training Interval   & Standard POD \(L^2\) Error & DDQ POD \(L^2\) Error \\ \hline
\([0,10]\) & 2.31E-12  & 1.04E-07 \\ \hline
\([0,5]\)  & 5.15E-12  & 4.94E-07 \\ \hline
\([0,1]\)  & 6.73E-11  & 1.25E-02 \\ \hline
\([0,.5]\) & 5.99E-02  & 4.49      \\ \hline
       \end{tabular}
  \end{center}
  \caption{Final Time \(L^2\) Error for Different Training Intervals when \(G = 0.001\) and \(r = 20\).}
\label{tab:5.10}
\end{table}

Both sets of data seem to indicate that for Standard POD there is a point somewhere between \(1/10\)th and \(1/20\)th of the main interval that the accuracy breaks down. The stability of the final time error is interesting for both cases as we are not only taking a shorter time interval we are also reducing the number of snapshots. 

This is not the case for the DDQ POD. The breakdown seems to occur at some point between \(T = 1 \) and \(T = 5\). More computations would be needed to have a better idea of the time which DDQ POD begins to struggle. 

\section{Conclusions}


We extended the DQ POD method proposed in \cite{Sarahs} to second difference quotients (DDQ) and proved data error formulas and pointwise data approximation error bounds. The POD data set for this method does not contain any redundant data and consists of one snapshot, one difference quotient, and all the second difference quotients. We considered the damped wave equation with viscous damping and Kelvin-Voigt damping to analyze the ROM errors using this DDQ POD method. Pointwise error bounds were developed when at least one damping parameter is nonzero.

We presented computational results for both standard POD and DDQ POD and the two types of damping. We presented results on the POD singular values and data error formulas for both methods and both types of damping. We gave data on the POD ROM maximum energy and pointwise errors for each damping parameter over a range of possible values. All computational results for the DDQ POD method followed the new theoretical results from this work.

The standard POD method was more accurate that the DDQ POD method in almost every numerical test; however, we do not have theoretical guarantees for the pointwise errors for the standard POD method. Preliminary experimentation inspired by \cite{Herkt} where alternative weights were used in the DDQ POD approach increased the accuracy of the method. More research in is needed to understand these observations. 

Finally, we explored using POD to simulate into the future. We compared using smaller test intervals to simulate across the entire interval of interest for each method of POD. The standard method once again performed better in this direction but the DDQ method was not far behind in performance. More work in this direction would be interesting and is left to be explored elsewhere. 

\section*{Acknowledgements}
  The authors thank the National Science Foundation (NSF) for providing support under grant number
2111421.

\bibliographystyle{plain}
\bibliography{SinglerJanes_PO_DDQ_preprint}
\end{document}